\documentclass[10pt]{amsart}

\usepackage[centertags]{amsmath}
\usepackage{amscd}
\usepackage{amsthm}
\usepackage{amssymb}
\usepackage[utf8]{inputenc}
\usepackage{todonotes} 
\usepackage{tikz}
\usepackage{verbatim}
\usepackage{booktabs}
\usepackage[colorlinks,backref]{hyperref} 
\usepackage{framed, paralist}

\theoremstyle{definition}
\newtheorem{Def}{Definition}[section]
\theoremstyle{plain}
\newtheorem{Lem}[Def]{Lemma}
\newtheorem{Cor}[Def]{Corollary}
	\newtheorem{Pro}[Def]{Proposition}
\newtheorem{Teo}[Def]{Theorem}
\newtheorem{Con}[Def]{Conjecture}
\newtheorem{TI}{Theorem}

\theoremstyle{remark}
\newtheorem{Rem}[Def]{Remark}
\newtheorem{Exa}[Def]{Example}

\makeatletter
\@addtoreset{table}{section}
\makeatother

\newcommand{\ord}{\operatorname{ord}}

\newcommand{\Spec}{\operatorname{Spec}}

\newcommand{\vol}{\operatorname{vol}}
\newcommand{\NP}{\operatorname{NP}}
\newcommand{\lex}{\rm{lex}}
\newcommand{\mult}{\operatorname{mult}}
\newcommand{\cent}{\operatorname{centre}}
\newcommand{\Proj}{\operatorname{Proj }}

\newcommand{\slo}{\operatorname{sl}}
\newcommand{\C}{{\mathbb C }}
\newcommand{\R}{{\mathbb R }}

\newcommand{\Z}{\mathbb{Z}}
\newcommand{\Q}{\mathbb{Q}}

\newcommand{\A}{\mathbb{A}}
\renewcommand{\P}{\mathbb{P}}
\newcommand{\I}{{\cal I}}

\renewcommand{\O}{{\mathcal O}}

\newcommand{\M}{\mathfrak m}

\newcommand{\p}{{\mathfrak p}}

\newcommand{\Mor}{ \overline{\rm NE}}

\renewcommand\hat{\widehat}
\newcommand{\NO}{{Newton--Okounkov }}
\newcommand{\deq}{\ensuremath{\stackrel{\textrm{def}}{=}}}
\newcommand{\equ}{\ensuremath{\,=\,}}
\newcommand{\codim}{{\rm codim}}
\newcommand{\st}[1]{\ensuremath{\left\{ #1\right\}}}

\makeatletter
\def\vv{\@ifnextchar[{\@withv}{\@withoutv}}
\def\@withv[#1]#2#3{v_1(#2,#3;#1)}
\def\@withoutv#1#2{v_1(#1,#2)}
\makeatother

\makeatletter
\def\mm{\@ifnextchar[{\@withm}{\@withoutm}}
\def\@withm[#1]#2#3{\mu_{#1}(#2,#3)}
\def\@withoutm#1#2{\hat\mu(#1,#2)}
\makeatother

\begin{document}

\title{Newton--Okounkov bodies sprouting on the valuative tree}

\author[Ciliberto] {Ciro Ciliberto}
\address{Dipartimento di Matematica, Universit\`a di Roma Tor Vergata, Italy}
\email{cilibert@axp.mat.uniroma2.it}

\author[Farnik] {Michal Farnik}
\address{Jagiellonian University, Faculty of Mathematics and Computer Science, {\L}ojasiewicza~6, 30-348 Krak\'ow, Poland}
\email{michal.farnik@gmail.com}

\author[K\"uronya] {Alex K\"uronya}
\address{Goethe-Universit\"at Frankfurt am Main, Institut f\"ur Mathematik, Robert-Mayer-Str. 6-10, D-60325 Frankfurt am Main, Germany}
\email{kuronya@math.uni-frankfurt.de}

 \author[Lozovanu]{Victor Lozovanu}
 \address{Universit\'e de Caen Normandie, Laboratoire de Math\'ematiques ''N. Oresme``, Campus C\^ote de Nacre, Boulevard du Mar\'echal Juin, 14032, Caen, France}
 \email{victor.lozovanu@gmail.com}

\author[Ro\'e]{Joaquim Ro\'e}
\address{Departament de Matem\`atiques, Universitat Aut\`onoma de Barcelona,
Edifici C, Campus de la UAB, 08193 Bellaterra (Cerdanyola del
Vall\`es)} \email{jroe@mat.uab.cat}

\author[Shramov]{Constantin Shramov}
\address{Steklov Institute of Mathematics, 8 Gubkina street, Moscow 119991, Russia;
National Research University Higher School of Economics, Russia}
\email{costya.shramov@gmail.com}


\date{}

\begin{abstract}
Given a smooth projective algebraic surface $X$, a point $O\in X$ and
a big divisor $D$ on $X$, we
consider the set of all Newton--Okounkov bodies of $D$ with respect to valuations of the field of rational functions of $X$ centred at $O$, or, equivalently, with respect to a flag $(E,p)$ which is \emph{infinitely near} to $O$, in the sense that there is a sequence of  blowups $X' \to X$, mapping the smooth, irreducible rational curve $E\subset X'$ to $O$. The main objective of this paper is to start a systematic study of the variation of these \emph{infinitesimal} Newton--Okounkov bodies as $(E,p)$ varies,
focusing on the case $X=\P^2$.
\end{abstract}

\maketitle

\tableofcontents

\section{Introduction}

The concept  of \emph{Newton--Okounkov bodies} originates in Okounkov's work \cite{Ok96}. Relying on earlier work of Newton and Khovanskii, Okounkov associates convex bodies to ample line bundles on
homogeneous spaces from a representation-theoretic point of view. In the generality we know them today, Newton--Okounkov bodies  have been introduced by Lazarsfeld--Musta\c t\u a \cite{LazMus09} and
Kaveh--Khovanskii \cite{KK09}.

Given an irreducible normal projective variety $X$ of dimension $r$ defined over an algebraically closed field $K$ of characteristic 0, a big divisor $D$ and a maximal rank valuation $v$ on the function field $K(X)$
(or, equivalently, an \emph{admissible/good  flag of subvarieties} on some proper birational model of $X$ (see \S\ref {ss:val})), we attach to these data a convex  body $\Delta_v(D)$ which encodes  in its convex geometric structure
the asymptotic vanishing behaviour  of the linear systems $|dD|$ for  $d\gg 0$ with respect to $v$.

Newton--Okounkov  bodies contain a lot of information: from a conceptual point of view, they serve as a set of `universal numerical invariants' according to a result of Jow~\cite{J}. From a more practical
angle, they reveal information about the structure of the Mori Cone of $X$ or of its blowups, about positivity properties of divisors  (ampleness, nefness, and the like,  see for instance Theorem \ref {thm:LK},
Remark~\ref {rem:LK} and~\cite{KL2}), and invariants like the volume or Seshadri constants (see~\cite{KL2, KL3}).

Not surprisingly,  the determination  of \NO bodies is extremely complicated in dimensions three and above. They can be non-polyhedral even if $D$ is ample and $X$ is a \emph{Mori dream space}
(see \cite {KLM}). We point out that the shape of $\Delta_v(D)$ depends on the choice of $v$ to a large extent: an adequate choice of a valuation can guarantee a more regular Newton--Okounkov
body \cite{AKL}.  The case of surfaces, though not easy at all, is reasonably  more tractable: the Newton--Okounkov bodies  are polygons with rational slopes, and they can be computed in terms of Zariski decompositions
(see \S\ref {ssec:no}).

In this paper we are mainly interested in \emph{infinitesimal \NO bodies}, which arise from valuations determined by flags $(E,p)$, with $p\in E$,
which are \emph{infinitely near} to  a point of the surface $X$: i.e., there is a birational morphism~\mbox{$X' \to X$} mapping the smooth, irreducible rational curve $E\subset X'$ to $O$. These \NO bodies have already been studied  in
\cite {KL2, KL3},  and their consideration is implicit in \cite{DHKRS}.  Here we intend to connect  the discussion in  \cite{DHKRS} to  infinitesimal \NO bodies.

One of the main underlying ideas of  \cite{DHKRS} is to study a certain invariant $\hat\mu$ (see~\S\ref{ssec:mu} for a definition),
which is roughly speaking an asymptotic multiplicity for quasi-monomial valuations. As such, it can be interpreted
as a function on the topological space~$\mathcal {QM}$, the \emph{valuative tree of quasi-monomial valuations}. Such  spaces were originally considered in the celebrated work of Berkovich  \cite {Ber90}, see also \cite{FJ04}). The tree $\mathcal{QM}$ is rooted, and  the root corresponds to the \emph {multiplicity valuation} centred at $O$, with infinite maximal arcs homeomorphic to $[1,\infty)$ starting from the root, and arcs sprouting from vertices corresponding to integer points (see  \cite {FJ04} and Remark \ref {tree}). The function $\hat \mu$ is  continuous along the arcs of $\mathcal {QM}$.
Interestingly enough, infinitesimal \NO bodies can be interpreted  as 2-dimensional  counterparts of~$\hat\mu$.

Here we will focus on the case  $X=\mathbb P^ 2$; the same questions on  other surfaces (general surfaces of degree $d$ in $\mathbb P^3$ for instance) are likely to   be equally interesting, but we do not treat them
in this work  in the hope that we will come back to them  in the future. A basic  property of $\hat\mu$, pointed out in  \cite{DHKRS}, is that
\[
\hat\mu(s)\,\geqslant\, \sqrt s
\]
assuming that   $s\in[1,+\infty)$ is an appropriately chosen parameter on an arc of $\mathcal {QM}$. Furthermore, equality holds unless there is a good geometric reason for the contrary, in the form of
a curve $C_s$ on $X_s$ (for $X_s$ the appropriate minimal blow-up of $X$ where the related flag  shows up) such that  the corresponding valuation takes a value higher  than~\mbox{$\deg (C_s)\cdot \sqrt s$}.
Such a curve is called \emph {supraminimal}  (see \S\ref {sec:supraminimal}).
Supraminimal curves are geometrically very particular, and give
information on the Mori cone of~$X_s$; for instance,
it is   conjectured in  \cite{DHKRS} (see also Conjecture~\ref {con:nn} below)
that along \emph{sufficiently general}
arcs of $\mathcal {QM}$, all supraminimal curves are
$(-1)$-curves. If so,
\[
\hat\mu(s) \equ \sqrt{s}  \ \ \text{for every $s\geqslant 8+1/36$},
\]
which among others implies  Nagata's celebrated  conjecture claiming that  the inverse of the $t$-point Seshadri constant of $\mathbb P^2$ equals $\sqrt t$ for  $t\geqslant 9$.

Once we fix a line $D$ in the plane, the convex geometric behaviour of the \NO body $\Delta_v(D)$  is essentially simple (it is a certain triangle) whenever  $\hat\mu$ has the expected
value (see \S\ref {ss:gen}). Otherwise $\Delta_v(D)$  exhibits more complicated features.  The interesting phenomenon which we study here is that, while $\hat \mu$ is continuous on~$\mathcal {QM}$,
the corresponding \NO bodies are not. Discontinuity takes the form that  at certain  vertices of the tree $\mathcal {QM}$, sudden discontinuous \emph{mutations} (even infinitely many of them) occur (see \S\ref{sec:supraminimal}).

Taking  into account the relation  between  \NO bodies and variation of Zariski decompositions (see Theorem \ref {thm:LM} and \cite[Theorem 1]{BKS}), the discontinuity phenomenon described above is to be
expected, since the same occurs for Zariski decompositions  in the big cone (Remark \ref {rem:disc}). The concept of a mutation is given   in  Definition \ref {def:mut} below.
The main object of our interest is the study of mutations  when the valuations move away from the root of $\mathcal {QM}$ along a \emph{fairly general} route,  and the results we have been able to obtain are collected
in \S\ref {sec:tt}. Our results  are partial  in the sense that there are intervals  in which we have been  unable to obtain  the appropriate information about mutations occurring there.
Our manuscript  is far from  conclusive,  it is simply devoted to lay the ground for future research on the subject.

We finish this introduction by observing that the general idea which lies behind all this is the following. There should be a \emph{Berkovich universal nef (or Mori) cone} $\mathcal B$
 with a continuous map $\beta: \mathcal B\to \mathcal {QM}$ such that  given a (rational) point $v\in \mathcal {QM}$, its preimage $\beta^{-1}(v)$ is the nef (or Mori) cone of the (minimal) blowup of $X$ where the flag~\mbox{$(E,p)$} related to $v$
 shows up. The variation of  \NO bodies when moving along the arcs of $\mathcal {QM}$ should provide information about $\mathcal B$.

The paper is organized as follows. In \S\ref {sec:preliminaries} we collect some basic definitions and results about valuations and \NO bodies, which we recall here to make the paper as self contained as possible.
In \S\ref {sec:valdim2} we focus on the two dimensional case, and specifically on quasi-monomial valuations, their interpretation in terms of the classical Newton--Puiseux algorithm, and the related  clusters of centres.
In this section (precisely in Remark \ref {tree}) we briefly recall the structure of the \emph{valuation tree} $\mathcal {QM}$. In \S\ref{sec:tt} we provide our computations about the infinitesimal \NO bodies.

In what follows we will mainly work over the field of complex numbers.

\section*{Acknowledgements}

This research was started during the workshop ``Recent advances in Linear series and
Newton--Okounkov bodies'', which took place in Padova, Italy, February 9--14, 2015. The authors enjoyed the lively and stimulating atmosphere of  that event.


\section{Preliminaries}
\label{sec:preliminaries}

 \NO bodies in the projective geometric setting have  been treated in \cite{LazMus09}, hence  this is the source we will primarily follow.

Let $X$ be an irreducible normal projective variety of dimension $r$ defined over an algebraically closed field $K$ of characteristic 0 (we will usually have the case $K=\C$ in mind), and
let $D$  be a big Cartier divisor  (or  line bundle; we may abuse terminology and identify the two concepts) on $X$.

Although one first  introduces Newton--Okounkov bodies for Cartier divisors,  the notion is numerical, even better,  it extends to big classes in  $N^1(X)_\R$ (see \cite[Proposition 4.1] {LazMus09}).
Newton--Okounkov bodies are defined with respect to a \emph{rank $r$ valuation} of the field of rational functions $K(X)$ of $X$.  We refer to  \cite[Chapter VI and Appendix 5]{ZS75} and
\cite[Chapter 8]{Cas00} for the general theory of valuations.

\subsection{Basics on Valuations}\label{ss:val}
\begin{Def}
A \emph{valuation} on $K(X)$ is a map $v:K(X)^*\rightarrow G$ where
$G$ is an ordered abelian group satisfying the following properties:
\begin{enumerate}
 \item $v(fg)=v(f)+v(g), \, \forall f, g\in K(X)^*$,
 \item $v(f+g)\geqslant \min(v(f),v(g)), \, \forall f, g\in K(X)^*$,
 \item $v$ is surjective,
 \item $v(a)=0,  \, \forall a \in K^*$.
\end{enumerate}
We will, as usual, assume that $v$ is surjective, in which
case $G$ is called the \emph{value group} of the valuation.
Two valuations $v, v'$ with value groups $G, G'$ respectively are said to be
\emph{equivalent} if there is an isomorphism $\iota:G\rightarrow G'$ of ordered groups
such that $v'=\iota\circ v$.
\end{Def}

The subring
$$R_v=\{f\in K(X)\,|\,v(f)\geqslant 0\}$$
is a \emph{valuation ring}, i.e., for all $f\in K(X)$,
 if $f\not\in R_v$  then $f^{-1}\in R_v$;
 its unique maximal ideal is $\M_v=\{f \in K(X) \,|\, v(f)>0\}$ and
the field $K_v=R_v/\M_v$ is called the \emph{residue field} of $v$.
Two valuations $v, v'$ are equivalent if and only if
$R_v=R_{v'}$ \cite[VI, \S8]{ZS75}.

\begin{Rem} Let $L$ be a line bundle on $X$. Given $f\in H^0(X,L)-\{0\}$,  it can be seen as a non-zero element in $K(X)$, therefore  one can consider $v(f)$. If $D=(f)$, then one sets $v(D):=v(f)$.
Thus, valuations can be considered to assume values on divisors;
effective divisors take nonnegative values.
\end{Rem}

\begin{Def}The \emph{rank} of a valuation $v$  is the minimal non-negative integer~$r$ such that the value group is isomorphic to an ordered subgroup of $\mathbb{R}^r_{\lex}$ (i.e. $\mathbb R^r$ with the \emph{lexicographic} order).
One can then write
\[
v(f)=(v_1(f),v_2(f),\dots,v_r(f))
\]
with $v_i:K(X)^*\rightarrow  \R$ for every integer
$i$ with $1\leqslant i\leqslant r$.

For every integer
$i$ with $1\leqslant i\leqslant r$, the \emph{$i$-th truncation} of $v$
is the rank $i$ valuation
\[
v|_i(f)=(v_1(f),v_2(f),\dots,v_i(f)).
\]

The \emph{trivial valuation}, defined as $v(f)=0$ for all $f\ne 0$, has rank zero;
it can be considered as the $0$-th truncation of all valuations $v$.
\end{Def}

\begin{Rem}
The rank of every valuation on $K(X)$ is bounded by $r=\dim (X)$,
and every valuation
of maximal rank is \emph{discrete}, i.e., it has a  value group  isomorphic to
$\Z^{r}_{\lex}\subset \mathbb{R}^{r}_{\lex}$
\cite[VI, \S10 and \S14]{ZS75}.
Whenever $v$ is a valuation of maximal rank,
one may assume that  the value group of $v$
\emph{equals} $\Z^{r}_{\lex}$ up to equivalence under  the action of some order-preserving (i.e., lower-triangular) element of ${\rm GL}(r,\mathbb R)$.
\end{Rem}

\begin{Rem} The \emph{rational rank}  of the valuation $v$ is the dimension of the $\mathbb Q$-vector space $G\otimes_\mathbb Z\mathbb Q$, where $G$ is the value group of $v$ (see \cite[\S10, p. 50]{ZS75}).
A valuation $v$ can be of rank 1 and rational rank $r>1$. The standard example is in \cite[\S14, Example 1, p. 100]{ZS75} (see also Remark \ref {rem:qm} below).
\end{Rem}

By \cite[VI, \S10, Theorem 15]{ZS75}, the rank of a valuation $v$ equals the Krull
dimension of its valuation ring $R_v$. More precisely, the ideals in $R_v$ are
totally ordered by inclusion, and if the rank is $r$, then the prime ideals of
$R_v$ are
\[
0=\p_0\subsetneq \p_1 \subsetneq \dots\subsetneq \p_r=\M_v, \quad \text{where}\quad \p_i=\{f \in R_v \,|\, v|_i(f)>0\}.
\]
The valuation rings of the truncations satisfy reverse inclusions
\[
K(X)=R_{v|_0}  \supsetneq R_{v|_1}  \supsetneq \dots\supsetneq R_{v|_r}=R_v
\]
as they are the localizations $R_{v|_i}=(R_v)_{\p_i}$.

By the valuative criterion of properness \cite[II, 4.7]{HAG},
since $X$ is projective, there is a (unique)
morphism
$$\sigma_{X,v}:\Spec (R_v)\rightarrow X$$
which, composed with
$\Spec (K(X)) \rightarrow \Spec( R_v)$, identifies $\Spec (K(X))$
as the \emph{generic point} of $X$.
 The image in $X$ of the closed point of $\Spec (R_v)$
(or the irreducible subvariety which is its closure)
is called the \emph{centre} of $v$ in $X$, and we denote it by $\cent_X(v)$. When the variety $X$
is understood, we shall write $\cent(v)=\cent_X(v)$. A valuation $v$ of rank $r$
determines a flag
\begin{equation}\label{eq:flag}
 X=\cent(v|_0) \supsetneq \cent(v|_1) \supseteq
\dots \supseteq \cent(v|_r) =\cent(v),
\end{equation}
and $\cent(v|_i)=\overline{\sigma_{X,v}(\p_i)}$.
Note that some of the inclusions may be  equalities.

For a valuation of rank $r>1$,  the centre
of the first truncation $v|_1$ is called the \emph{home} of $v$,
following \cite{BFJ08}.

\begin{Exa} (Divisorial valuations)\label{divisorial}
  If $\cent(v)$ is a divisor $V$, then $v$ is equivalent to the valuation that assigns to each rational function its order of vanishing along~$V$. Moreover, the residue field $K_v$ is
the function field of $V$ (see \cite[VI, \S14]{ZS75}).\end{Exa}

\begin{Exa}(Valuation associated to an admissible flag)\label{rem:flag} A \emph{full flag} $Y_\bullet$ of irreducible subvarieties
\begin{equation}\label{eq:flag2}
 X \equ Y_0 \supset Y_1 \supset \ldots\supset Y_{r-1} \supset Y_r
\end{equation}
is called \emph{admissible}, if $\codim_X (Y_i)=i$ for all $0\leqslant i\leqslant\dim (X)=r$, and $Y_i$ is normal and smooth at the point $Y_r$, for all $0\leqslant i\leqslant r-1$. The flag is called \emph{good} if $Y_i$ is smooth for all $i=0,\ldots,r$.

 Let $\phi\in K(X)$ be a non-zero rational function, and set
 \[
  \nu_1(\phi) \deq \ord_{Y_1}(\phi) \quad \text{and}   \quad  \phi_1 \deq \left.\frac{\phi}{g_1^{\nu_1(\phi)}}\right|_{Y_1}
 \]
where $g_1=0$ is a local equation of $Y_1$ in $Y_0$ in an open Zariski subset around the point $Y_r$.  Continuing this way via
\[
 \nu_i(\phi) \deq \ord_{Y_i}(\phi_{i-1})\ ,\  \phi_i \deq \left.\frac{\phi_{i-1}}{g_i^{\nu_i(\phi_{i-1})}}\right|_{Y_i} \quad \text {for all}\quad i=2,\ldots, r,
\]
where $g_{i}=0$ is a local equation of $Y_i$ on $Y_{i-1}$ around $Y_r$,
we arrive at a function
\[
 \phi \mapsto \nu_{Y_\bullet}(\phi) \deq (\nu_1(\phi),\dots, \nu_r(\phi))\ .
\]
One verifies that $\nu_{Y_\bullet}$ is a valuation of maximal rank, and that the flag \eqref{eq:flag} given by
the centres of the truncations of $\nu_{Y_\bullet}$ coincides with the flag $Y_\bullet$ in \eqref {eq:flag2}.
\end{Exa}

\begin{Pro}\label{iterated_divisorial}
Let $v$ be a valuation of maximal rank $r=\dim (X)$ whose flag of centres $Y_\bullet$
in  \eqref {eq:flag}
is admissible.
Then $v$ is equivalent to the flag valuation $\nu_{Y_\bullet}$.
\end{Pro}
\begin{proof}

By induction on $r$. For $r=0$ there is nothing to prove, so assume $r\geqslant 1$.
Remark \ref{divisorial} tells us that
$v|_1=\nu_{Y_\bullet}|_1$ (up to equivalence),
and that their common residue field is $K(Y_1)$, with
$Y_1=\cent_{X}(v|_1)$.
The valuation $v$ (resp. $\nu_{Y_\bullet}$) induces a valuation
$\bar v$ (resp. $\bar\nu_{Y_\bullet}$) on
$K(Y_1)=R_{v|_1}/\M_{v|_1}$ as follows. For any
$$0\neq \bar f\in R_{v|_1}/\M_{v|_1},$$
there is an $f\in K(X)$ sitting in $R_{v|_1}$
whose class modulo $\M_{v|_1}$ is $\bar f$. Then one sets
$\bar v(\bar f)=v(f)$ (similarly for $\bar\nu_{Y_\bullet}$) and one verifies that this is well defined. The valuations $\bar v$ and $\bar\nu_{Y_\bullet}$ have maximal
rank $r-1$, and their flag of centres is $\bar Y_\bullet$ for both, with
$\bar Y_i=Y_{i+1}$ for $i=0, \dots, r-1$, so by  induction
they are equivalent.

Finally, the valuation ring of $v$ (resp. of $\nu_{Y_\bullet}$) consists  of
those $f$ in $K(X)$ with $v|_1(f)=\nu_{Y_\bullet}|_1(f)>0$ and of those
$f$ in $K(X)$ with $v|_1(f)=\nu_{Y_\bullet}|_1(f)=0$ and $\bar v(\bar f)\geqslant 0$ (resp. $\bar\nu_{Y_\bullet}(\bar f)\geqslant 0$).
Since $\bar v$ and $\bar\nu_{Y_\bullet}$ are equivalent, then the
valuation rings of $v$ and $\nu_{Y_\bullet}$ are the same, as claimed.
\end{proof}

Valuations of maximal rank are very well known
(see \cite[VI, \S14]{ZS75},
 \cite[Examples 5 and 6] {Vaq}) and Theorem \ref{flag_valuation} below
is presumably obvious for  experts working in the area of   resolution of singularities.
We include a proof as we lack a precise reference for it.
For the case of surfaces, see \S\ref{flagmodel} below.

\begin{Teo} \label{flag_valuation}
 Let $X$ be a normal projective variety, and $v$ a valuation
of the field $K(X)$ of maximal rank $r=\dim (X)$. There exist a proper birational morphism
$\pi:\tilde X \rightarrow X$  and a good flag
\[
Y_\bullet : \tilde X= Y_0 \supset Y_1 	\supset \dots \supset Y_r
\]
such that $v$ is equivalent to the valuation associated to $Y_\bullet$.
\end{Teo}

\begin{proof} Denote by $\zeta\in X$ the generic point,
and set $\mathbb K=K(\zeta)=K(X)$.
Let
\[
0 =\p_0\subset \p_1 \subset \dots\subset \p_r
\]
be the maximal chain of prime ideals in $R_v$,
and choose $f_1, \dots, f_r \in R_v\subset \mathbb K$ so that each
$f_i\in \p_i\setminus \p_{i-1}$.
Fix projective coordinates $[x_0:\dots:x_r]$ in
$\P^r_\mathbb K \subset \P^r_K \times X$, and let 
$\xi=[1:f_1: \dots: f_r]\in \P^r_\mathbb K$.
Let $X_0$ be the Zariski closure of $\xi$ in $\P^r_K \times X$.
Since its generic
point is $\xi$ (which is a closed $\mathbb K$-point in $\P^r_{\mathbb K}$),  it has residue field equal to $\mathbb K$, and the induced projective morphism $X_0 \rightarrow X$ is
birational.

For $i=1,\ldots,r$, the restriction of the rational function $x_i/x_0$ to $X_0$ is $f_i$, which has positive
$v$-value. Therefore the centre of $v$ in $X_0$ lies in $[1:0:\dots:0]\times X$, and its
local ring $\O_{\cent_{X_0}(v)}$ contains $f_1, \dots, f_r$. Hence
\[f_i \in \left(\p_i \cap \O_{\cent_{X_0}(v)}\right) \setminus
\left( \p_{i-1} \cap \O_{\cent_{X_0}(v)}\right),\]
so that $\p_i \cap \O_{\cent_{X_0}(v)} \ne \p_{i-1} \cap \O_{\cent_{X_0}(v)}$. Since
$\p_i \cap \O_{\cent_{X_0}(v)}=\sigma_{X_0,v}(\p_i)$ as schematic points in $X_0$,
it follows that the centres of the truncations of $v$ are all distinct. Since there
are as many truncations as the dimension of $X$, the flag
\eqref{eq:flag} in $X_1$ is a full flag, i.e., $\dim (\cent_{X_0}(v|_i))=r-i$
for $i=0,\dots,r$.
Every birational model of $X$ dominating $X_0$ will again have this property.

The flag of centres of the truncations in
$X_0$ is usually not good (or even admissible), as $X_0$ is not necessarily
smooth (not even normal) at $\cent_{X_0}(v)$.
Using Hironaka's
resolution of singularities we know that there is a birational
morphism  $X_1 \rightarrow X_0$, obtained as a composition of blowups
along smooth centres, with $X_1$ a smooth projective variety.  On
$X_1$ we have a full flag like \eqref{eq:flag} whose codimension 1 term,
$\cent_{X_1}(v|_1)$, may be singular. But again there is a
composition of blowups along smooth centres (contained in $\cent_{X_1}(v|_1)$)
that desingularizes it; we apply these blowups to $X_1$, to obtain
$X_2 \rightarrow X_1$. Since the blowup of a smooth variety along a smooth
centre is again smooth, $X_2$ stays smooth, and the divisorial part of
the full flag \eqref{eq:flag} in $X_2$ is now also smooth. By resolving
sequencially the singularities of $\cent_{X_2}(v|_2)$, \dots, $\cent_{X_{r-1}}(v|_{r-1})$
we arrive at a model $\tilde X=X_r$ where the flag
\[
 Y_\bullet:\tilde X=\cent_{\tilde X}(v|_0) \supsetneq \cent_{\tilde X}(v|_1) \supsetneq
\dots \supsetneq \cent_{\tilde X}(v|_r) =\cent_{\tilde X}(v),
\]
is good. Now
by Proposition \ref{iterated_divisorial}, the valuation $v$ is equivalent to
$\nu_{Y_\bullet}$ as claimed.
\end{proof}
\begin{Rem}
 We work here in characteristic 0, but a suitable (weaker)  version of Theorem
 \ref{flag_valuation} still holds in any characteristic. The same proof works, by replacing Hironaka's resolution with a sequence of blowups along nonsingular centres given by Urabe's \emph{resolution of maximal rank
 valuations} \cite{Ura}. The members of the resulting flag are not necessarily
 smooth, but they are non-singular at the centre.
\end{Rem}
In the situation of Theorem \ref{flag_valuation}, we call $Y_\bullet$ the \emph{good flag associated} to $v$ in the model
$\tilde X$. The choice of a flag is not unique, but for two models, the induced rational map between them maps the associated flags into one another.

\subsection{\NO bodies}

\begin{Def}
Let $X$ be an irreducible normal projective variety, $D$ a big divisor on $X$, and $v$ a valuation
of $K(X)$ of maximal rank $r=\dim (X)$. Define the \emph{\NO
body} of $D$ with respect to $v$ as follows
\begin{equation}\label{eq:OB}
\Delta_v(D) \deq\text{ convex hull}
\overline{\left\{\bigcup_{k\in\mathbb{Z}_{>0}}\left. \left\{\frac{v(f)}{k}\right| f\in H^0(X,\O_X(kD))-\{0\}\right\}\right\}}.
\end{equation}
The points in $\Delta_v(D)\cap \mathbb Q^r$ of the form  $\frac{v(f)}{k}$ with $f\in H^0(X,\O_X(kD))-\{0\}$ for some integer $k>0$ are called \emph{valuative points}.
 \end{Def}

\begin{Rem}\label{rem:val}
The properties of valuations yield that if $A$, $B$ are two distinct valuative points, then any rational point on the segment joining $A$ and $B$ is again a valuative point. This implies that valuative points are dense in $\Delta_v(D)$ (see \cite [Corollary 2.10]{KL}, for the surface case; the proof is analogous in general).  Therefore in \eqref{eq:OB} it suffices to take the closure in the Euclidean topology of  $\R^r$.

Alternatively, one defines the \NO body of $D$ with respect to $v$ as
\[
\Delta_{v}(D) \deq  \overline{v\big(\{D' \ | \ D'\equiv D \textup{ effective }\Q\textup{-divisor}\}\big)},
\]
where $\equiv$ is the $\Q$-linear equivalence relation.
By  \cite[Proposition~4.1]{LazMus09},  one can replace $\Q$-linear equivalence by  \emph{numerical equivalence}. Hence, one can define $\Delta_{v}(\zeta)$ for any numerical class $\zeta$ in the \emph{big cone} ${\rm Big}(X)\subset N^1(X)_\mathbb R$ of $X$.
\end{Rem}

Our definition  differs from the one in  \cite{LazMus09} in that we use valuations of maximal rank instead of those defined by admissible  flags on $X$. But, an admissible flag on $X$ gives rise to a valuation of maximal rank on $K(X)$ by Example \ref {rem:flag} (see also \cite{KK09}). Conversely, by Theorem  \ref{flag_valuation}, any valuation of maximal rank arises  from an admissible flag on a suitable proper birational model of $X$; thus maximal rank valuations are the birational version of admissible flags. In conclusion,
all known results for \NO bodies defined in terms of flag valuations carry over to \NO bodies in terms of valuations of maximal rank, modulo passing to some different birational
model.

In \cite{Bou12,KK09} one considers \NO bodies
defined by valuations
of maximal rational rank, an even more general situation
which we will not consider here.


\subsection{Some properties of \NO bodies}\label{ssec:no}

A very important feature of Newton--Okounkov bodies is that they give rise to a `categorification' of various asymptotic invariants associated to
line bundles (see for instance \cite[Theorem C]{KL} for the corresponding statement for moving Seshadri constants). Recall that  the  \emph{volume} of a
Cartier divisor $D$ on an irreducible normal projective variety $X$ of dimension $r$ is defined as
\[ \vol (D)\deq \limsup_{m\to\infty}\frac{\dim \big (H^0(X,\O_X(mD))\big )}{m^r/r!} \ .\]

\begin{Teo} [Lazarsfeld--Musta\c{t}\u{a}, {\cite[Theorem 2.3]{LazMus09}}]\label{thm:LM2}
 Let $X$ be an irreducible normal projective variety of dimension  $r$, let $D$ be a big divisor on $X$, and let $v$ be a valuation of the field $K(X)$
with value group $\Z^r_{\lex}$. Then
\[\vol (\Delta_v(D)) = \frac{1}{r!}\vol(D), \]
where the volume on the left-hand side denotes the Lebesgue measure
in $\mathbb{R}^r$.
\end{Teo}

\begin{Rem} Although the proof of Theorem~2.3 from \cite{LazMus09} takes the admissible flags viewpoint, the statement remains valid for \NO bodies defined in  terms of valuations of maximal rational rank  (with value group equal to
$\Z^r$) by the remark above  (see also \cite[Corollaire 3.9]{Bou12}).
\end{Rem}

Since the main focus of our work is on the surface case, we will concentrate on surface-specific properties of Newton--Okounkov bodies.

\begin{Teo}[K\"uronya--Lozovanu--MacLean, \cite{KLM}]
If $\dim (X)=2$, then every \NO body is a polygon.
\end{Teo}

If $\dim X=2$, then an admissible flag  is given by a pair $(C,x)$, where $C$ is a curve, and $x\in C$ a smooth point.
If $D$ is a big divisor on $X$, the corresponding Newton--Okounkov body will be denoted by $\Delta_{(C,x)}(D)$.

\begin{Rem}
In fact one can say somewhat more about the convex geometry of Newton--Okounkov polygons, see \cite[Proposition 2.2]{KLM}. First, all the slopes of its edges are rational. Second, if one defines
\[
\mu_C(D) \deq \sup \st{t>0\mid D-tC \text{ is big}}\ ,
\]
then all the vertices of $\Delta_{(C,x)}(D)$ are rational with possibly two exceptions, i.e. the points of this convex set lying on the line $\{\mu_C(D)\}\times \R$.
\end{Rem}

Lazarsfeld and Musta\c t\u a observe in \cite[Theorem 6.4]{LazMus09} that variation of Zariski decomposition \cite[Theorem 1]{BKS}
provides a  recipe for computing  Newton--Okounkov bodies in the surface case. Let $D=P+N$ be the Zariski decomposition of $D$ (for the definition and basic properties see \cite{Bad,FU}), where  the notation, here and later, is the standard one: $P$ is the \emph{nef part} ${\rm Nef}(D)$ and $N$ the
\emph{negative part} ${\rm Neg}(D)$ of the decomposition. Denote by $\nu=\nu(D,C)$ the coefficient of $C$ in $N$ and $\mu=\mu_C(D)$ whenever there is no danger of confusion.

For any $t\in[\nu ,\mu]$, set  $D_t=D-tC$ and let $D_t=P_t+N_t$  be the Zariski
decomposition of $D_t$. Consider the functions $\alpha,\beta: [\nu,\mu]\to \mathbb R^+$ defined as follows
\[
\alpha(t)\deq \ord_x(N_{t\vert C}), \quad \beta(t)\deq\alpha(t)+P_t\cdot C \ .
\]

\begin{Teo} [Lazarsfeld--Musta\c{t}\u{a}, {\cite[Theorem 6.4]{LazMus09}}]\label{thm:LM}
In the above setting one has
\[
\Delta_{(C,x)}(D) \equ \left\{(t,u)\in \mathbb R^2 \vert\,\, \nu\leqslant t\leqslant \mu, \,\, \alpha(t)\leqslant u\leqslant\beta(t)\right\}.
\]
\end{Teo}

\begin{Rem}
Note that all the results concerning Newton--Okounkov bodies use Zariski decomposition in Fujita's sense, i.e. for pseudo-effective $\R$-divisors.
\end{Rem}

As an immediate consequence we have:

\begin{Cor} \label{cor:LM}
In the above setting the lengths of the vertical slices of $\Delta_{(C,x)}(D)$ are independent of the (smooth) point
$x\in C$.
\end{Cor}

\begin{Rem}\label{rem:disc} (See  \cite[proof of Proposition 2.2]{KLM})
In the above setting, the function $t \to N_t$  is \emph{nondecreasing} on $[\nu,\mu]$, i.e. $N_{t_2}-N_{t_1}$ is effective whenever $\nu\leqslant t_1\leqslant t_2\leqslant \mu$. This implies that a vertex $(t,u)$ of $\Delta_{(C,x)}(D)$ may only occur for those $t\in [\nu,\mu]$ where the ray $D-tC$ crosses into a different Zariski chamber, in particular, where a new curve appears in $N_t$.
\end{Rem}

Given three real numbers $a>0, b\geqslant 0, c>0$, we will denote by $ \Delta_{a,b,c}$ the triangle with vertices $(0,0)$, $(a,0)$ and $(b,c)$. We set $\Delta_{a,c}:=\Delta_{a,0,c}$ and $\Delta _{a,a}:=\Delta_a$.
Note that the triangle $\Delta_{a,b,c}$ degenerates into a segment
if $c=0$.

\begin{Exa}\label{ex:p} In the above setting suppose that $D$ is an ample divisor. Then, by Theorem \ref {thm:LM}, the \NO body $\Delta_{(C,x)}(D)$ contains the triangle $\Delta_{\mu_C(D),D\cdot C}$, and by Theorem \ref {thm:LM2} one has
\[
\mu_C(D)\leqslant\frac {D^ 2}{D\cdot C}\ .
\]
Equality holds if and only if $\Delta_{(C,x)}(D)=\Delta_{\mu_C(D),D\cdot C}$. In particular, if $X=\mathbb P^ 2$, $C$ is a curve of degree $d$, and $D$ a line, then $\Delta_{(C,x)}(D)=\Delta_{\frac 1d,d}$.
\end{Exa}

Let $D$ be a big divisor on $X$. Let ${\rm Null}(D)$ be the divisor (containing ${\rm Neg}(D)$) given by the union of all irreducible curves $E$ on $X$ such that ${\rm Nef}(D)\cdot E=0$.

\begin{Teo}[K\"uronya--Lozovanu, {\cite[Theorem 2.4, Remark 2.5]{KL}}] \label{thm:LK}
Let $X$ be a smooth projective surface, $D$ be a big divisor on $X$ and $x\in X$ a
point. Then: \begin{itemize}
\item [(i)] $x\not\in {\rm Neg}(D)$ if and only if for any admissible flag $(C,x)$ one has $(0,0)\in \Delta_{(C,x)}(D)$;
\item [(ii)] $x\not\in {\rm Null}(D)$ if and only if for any admissible flag $(C,x)$ there is a positive number $\lambda$ such that $\Delta_\lambda\subseteq  \Delta_{(C,x)}(D)$.
\end{itemize}
\end{Teo}

\begin{Rem}\label {rem:LK} The divisor $D$ is nef (resp. ample) if and only if ${\rm Neg}(D)=\emptyset$ 
(resp. ${\rm Null}(D)=\emptyset$), so that Theorem \ref {thm:LK} provides nefness and ampleness criteria for $D$ detected from \NO bodies.

Note that Theorem \ref {thm:LK} has a version in higher dimension (see \cite {KL2}). The same papers \cite {KL,KL2} explain how to read the \emph{moving Seshardi constant}  of $D$ at a point $x\not\in {\rm Neg}(D)$ from \NO bodies.
\end{Rem}

\section{Valuations in dimension 2}\label{sec:valdim2}

\subsection{Quasimonomial valuations}\label{ssec:qmv}

We will mainly treat the  case $X=\mathbb P^2$ and $D$ a line, leaving to the reader to make the obvious adaptations for other surfaces.

Let $O$ denote the origin $(0,0)\in\A^2=\Spec (K[x,y]) \subset \P^2=\Proj (K[X,Y,Z])$ with $x=X/Z$, $y=Y/Z$, and let $\mathbb K=K(\P^2)=K(x,y)$
be the field of rational functions in two variables. We will focus on \NO bodies of  $D$ with respect
to rank 2 valuations $v=(v_1,v_2)$ with centre at $O$, with the additional condition
that either the home of $v$ is a smooth curve through $O$,
or it is equal to $O$ (in which case we call the
corresponding body an \emph{infinitesimal} \NO body) and $v_1$
is a  \emph{quasimonomial} valuation.

Fix a smooth germ of curve $C$   through $O$;  we can assume without
loss of generality that $C$ is tangent to the line $y=0$; hence  $C$ can be locally
parameterized by $x \mapsto (x,\xi(x))\in \A^2$, where $\xi(x)\in K[[x]]$  with $\xi(0)=\xi'(0)=0$.

\begin{Def}\label{seriesdef}
Given a real number $s\geqslant 1$ and any $f \in \mathbb K^*$, set
\begin{equation}\label{eq:c-exp1}
\vv[f]{C}{s} := \ord_x(f(x,\xi(x)+\theta x^s))\ ,
\end{equation}
where $\theta$ is transcendental over $\mathbb C$. Equivalently, expand $f$ as a Laurent series
\begin{equation}\label{eq:c-exp}
 f(x,y) = \sum a_{ij}x^i (y-\xi(x))^j \ .
\end{equation}
One has
\begin{equation}\label{monomialdef}
 \vv[f]{C}{s} = \min\{i+sj|a_{ij}\ne 0\}\ .
\end{equation}
Then $f\mapsto \vv[f]{C}{s}$ is a rank $1$ valuation which we denote by $\vv{C}{s}$.
Such valuations are called \emph{monomial} if $C$ is the line $y=0$
(i.e., $\xi=0$), and \emph{quasimonomial} in general. The point $O$ is the centre of the valuation.

We call \eqref{eq:c-exp} the \emph{$C$-expansion} of $f$. Slightly abusing language, $s$ will be called the
\emph{characteristic exponent} of $\vv{\xi}{s}$ (even if it is an integer).
\end{Def}

\begin {Exa} \label{ex:mult} The valuation $v_O:=\vv{C}{1}$ is the
 \emph{$O$-adic valuation} or \emph{multiplicity valuation}:  if $f$ is a non--zero polynomial, then $v_O(f)$ is the \emph{multiplicity} $\mult_O(f)$
of $f=0$ at $O$.
\end{Exa}

\begin{Rem}\label{rem:qm}
 The value group of $\vv{C}{s}$ is:
 \begin{itemize}
 \item $\Z \frac{1}{q}\subset \Q$
 if $s$ is a rational number $s=\frac pq$ with $\gcd(p,q)=1$;
 \item $\Z+\Z s \subset \R$ if $s$ is an irrational number.
 \end{itemize} So the rank of $\vv{C}{s}$ is 1, but
 in the latter case
 the valuation has \emph{rational rank 2}. We will be mostly concerned with
 the rational case. Note that $\vv{C}{s}$ is discrete if and only if $s$ is rational.
\end{Rem}

\begin{Rem}\label {rem:ocs}
The valuation $\vv{C}{s}$ depends only on the
$\lfloor s \rfloor$-th jet of $C$, so for fixed $s$ the series $\xi$ can be  assumed
to be a polynomial; however, later on we shall let $s$ vary for a fixed $C$, so we better keep $\xi(x)$ a series.
\end{Rem}

\begin{Exa}\label{ex:C}
If $f=0$ is the equation of $C$ (supposed to be algebraic, which, for fixed $s$ is no restriction by Remark \ref {rem:ocs}), then by plugging $y=\xi(x)$ in \eqref {eq:c-exp} we have $\sum_{i=1}^\infty a_{i0}x^ i\equiv 0$, hence $a_{i0}=0$, for all $i\geqslant 0$. Then
$f(x,y)=(y-\xi(x))\cdot g(x,y)$ where $g(0,0)\neq 0$. This implies that $v_1(C,s;f)=s$, which can be also deduced from \eqref {eq:c-exp1} by expanding $f(x,\xi(x)+\theta x^s)$ in Taylor series with initial point $(x,\xi(x))$.
\end{Exa}

\begin{Rem}\label{tree} (See \cite {FJ04}) The set $\mathcal {QM}$ of all quasi-monomial valuations
with centre at $O$  has a  natural topology, namely
the coarsest topology such that for all $f\in \mathbb K^*$,
$v\mapsto v(f)$ is a continuous map $\mathcal {QM} \rightarrow \R$.
This  is called the \emph{weak topology}. For a fixed
$C$, the map $s\mapsto \vv{C}{s}$ is  continuous in $[1,+\infty)$.

There is however a   finer topology of interest on the valuative tree $\mathcal {QM}$:  the finest
topology such that $s\mapsto \vv{C}{s}$ is continuous in $[1,+\infty)$ for all $C$.
This latter  is called the \emph{strong topology}. With the strong topology,
$\mathcal {QM}$ is a profinite $\R$-tree, rooted at the \emph{$O$-adic valuation} (see \cite{FJ04} for details).
To avoid confusion with branches of curves, we will call  the branches in $\mathcal {QM}$ \emph{arcs}.
Maximal arcs of the valuative tree are homeomorphic to the interval  $[1,\infty)$, parameterized by $s\mapsto \vv{C}{s}$ where $C$
is a smooth branch of curve at $O$.

The arcs  of $\mathcal {QM}$ share the  segments given by
coincident jets, and separate at integer values of $s$;
these correspond to divisorial valuations on an appropriate birational model.

Though we will not use this fact, note that  $\mathcal {QM}$ is a sub-tree of a larger $\R$-tree $\mathcal V$ with the same root, called  the
\emph{valuation tree}, which consists of all real valuations of $\mathbb  K$ with centre $O$. Ramification on $\mathcal V$ occurs at all \emph{rational}
points of the arcs, rather than only at integer points, because of
valuations corresponding to singular branches. The tree $\mathcal {QM}$
is obtained from $\mathcal V$ by removing
the arcs corresponding to singular branches and all
ends (see \cite[Chapter 4]{FJ04} for details).
\end{Rem}

\subsection{Quasimonomial valuations and the Newton--Puiseux algorithm}\label{ssec:NP}

We recall briefly the \emph {Newton--Puiseux algorithm} (see \cite[Chapter 1]{Cas00} for a full discussion).

Given $f(x,y)\in K(x,y)-\{0\}$ (we may in fact assume that $f$ belongs to  $K[[x,y]]$),  and a curve $C$ as in \S\ref {ssec:qmv}, we want to investigate
the behavior of the function
\[
v_1(C;f)\colon  s\in [1,+\infty) \mapsto v_1(C,s;f)\in \R\ .
\]
Returning to \eqref {eq:c-exp}, consider the convex hull ${\overline \NP}(C,f)$ in $\R^ 2$ (with $(u,v)$ coordinates) of all points $(i,j)+{\bf v} \in \R^ 2$
such that $a_{ij}\neq 0$, and ${\bf v}\in \R_+^ 2$. The boundary  of ${\overline \NP}(C,f)$ consists of two half-lines
parallel to the $u$ and $v$ axes, respectively, along with  a polygon $\NP(C,f)$, named the  \emph{Newton polygon} of $f$ with respect to $C$.

We will denote by $\mathfrak V(C,f)$ (resp. by $\mathfrak E(C,f)$) the set of vertices (resp. of edges) of $\NP(C,f)$,  ordered from left to right, i.e.,
\[
\mathfrak V(C,f) \equ (\mathfrak v_0,\ldots, \mathfrak v_h)\ ,\ \quad \mathfrak E(C,f)\equ (\mathfrak l_1,\ldots, \mathfrak l_h)\ ,
\]
where $\mathfrak l_k$ is the segment joining $\mathfrak v_{k-1}$ and $\mathfrak v_k$, for $k=1,\ldots,h$, and $\mathfrak v_k=(i_k,j_k)$.

We will denote by $V$ the germ of the curve $f=0$. Then,
\[
\mult_O (V) \equ \textup{min}_{k}\{ i_k+j_k\}, \ \ord_C (V) \equ j_h,\  ((V-j_hC),C)_O \equ i_h\ ,
\]
where $((V-j_hC),C)_O$ is the \emph{local intersection number} of two effective cycles with distinct support $V-j_h C$ and $C$  at the origin $O$.

The numbers $n:=w(C,f):=i_h-i_0$ and $m:=h(C,f):=j_0-j_h$
are usually called the \emph{width} and the \emph{height} of $\NP(C,f)$. Analogously, one defines the \emph{width}
$n_k:=w(\mathfrak l_k)$ and the \emph{height} $m_k:=h(\mathfrak l_k)$ for any edge $\mathfrak l_k$ in the obvious way, so that $\mathfrak l_k$ has
\emph{slope} $s_k:=\slo(\mathfrak l_k)=- \frac {m_k}{n_k}$, for $k=1,\ldots, h$.

Let $\mathcal B(V)$ be the set of \emph{branches} of $V$. Then the Newton--Puiseux algorithm as presented in
\cite[\S1.3]{Cas00}
(with suitable modifications due to the fact that \eqref {eq:c-exp} is not the standard expansion of $f(x,y)$ as a power series in $x$ and $y$) yields  a surjective map
\[
\varphi_V: \mathcal B(V)\to \mathfrak E(C,f)
\]
such that whenever $\gamma$ is a branch of $V$ whose Puiseux expansion
with respect to $C$ starts as
\[
y-\xi(t)=ax^t+\ldots, \textup{ with } t\in \Q \textup{ and }t\geqslant 1
\]
(i.e., $\gamma$ is not tangent to the $x=0$ axis nor contained
in $C$)
then the edge $\mathfrak l=\varphi_V(\gamma)$ has slope  $\slo(\mathfrak l)=- \frac 1 t\geqslant -1$.
Moreover
$$t=\frac {(\gamma,C)_O}{\mult_O(\gamma)}\ ,$$
and if $\gamma\in \mathcal B(V)$ is the unique branch with $\slo(\varphi_V(\gamma))=-\frac 1 t$, then in fact $h(\varphi_V(\gamma))=\mult_O(\gamma)$ is the \emph{multiplicity} of $\gamma$ (at $O$),
whereas $w(\varphi_V(\gamma))=(\gamma,C)_O$ is the \emph{local intersection multiplicity} of $\gamma$ with $C$ at $O$.

Consider now the line $\ell_s$ with equation $u+sv=0$ and slope $-\frac 1s$. By \eqref {monomialdef}, the valuation  $v_1(C,s;f)$ is computed by those vertices in $\mathfrak V(C,f)$ with the smallest distance to~$\ell_s$, i.e. for any such vertex $\mathfrak v=(i,j)$, one has
$v_1(C,s;f) \equ i+sj$. Note that there will be only one such point, unless  $\ell_s$ is parallel to one of the edges $\mathfrak l\in \mathfrak E(C,f)$
(hence $s$ is rational), in which case there will be  two: the vertices of $\mathfrak l$, whose slope $\slo(\mathfrak l)=-\frac 1s$.

From the above discussion its not hard to deduce the following statement:

\begin{Pro}\label{lem:trop}
 For any curve $C$ smooth at $O$ and $f\in K(x,y)-\{0\}$ \emph{regular} at~$O$ (i.e., $f$ is defined at $O$) one has:\\
 \begin{inparaenum}
 \item [(i)] $v_1(C,\cdot\,;f):\R\rightarrow\R$ is continuous in $[1,+\infty)$, piecewise linear, non-decreasing,  concave and its graph consists of  finitely many (one more than the number of edges in $\mathfrak E(C,f)$ with slope $t>-1$) linear arcs with rational slopes (i.e., $v_1(C,\cdot\,;f)$ is a \emph{tropical polynomial} in $s$); \\
 \item [(ii)] the points where  the derivative of $v_1(C,\cdot\,;f)$ is not defined are
$$s_k=-\frac 1 {\slo(\mathfrak l_k)}, \quad \text{for} \quad k=1,\ldots, h;$$\\
 \item [(iii)] if the curve $V$ with equation $f=0$ does not contain $C$, then
 \begin {equation}\label{eq:const}
v_1(C,s;f)= {(V,C)_O}  \quad \text{for} \quad s\gg 1.
\end{equation}
\end{inparaenum}\end{Pro}

\begin{Exa} \label{exa:Newton}
Let $C$ be the conic $x^2-2y=0$, so that $\xi(x)=x^2/2$, and let
$$f=(x^2+y^2)^3-4x^2y^2.$$
The $C$-expansion of $f$ is then
\begin{multline*}
 f=\left(y-\xi\left(x\right)\right)^6+3\,x^2\,\left(y-\xi\left(x
 \right)\right)^5+{{15\,x^4\,\left(y-\xi\left(x\right)\right)^4
 }\over{4}}+3\,x^2\,\left(y-\xi\left(x\right)\right)^4\\
 +{{5\,x^6\,\left(y-\xi\left(x\right)\right)^3}\over{2}}+6\,x^4\,\left(y-\xi
 \left(x\right)\right)^3+{{15\,x^8\,\left(y-\xi\left(x\right)\right)^
 2}\over{16}}+{{9\,x^6\,\left(y-\xi\left(x\right)\right)^2}\over{2}}\\
 +3\,x^4\,\left(y-\xi\left(x\right)\right)^2-4\,x^2\,\left(y-\xi\left(
 x\right)\right)^2+{{3\,x^{10}\,\left(y-\xi\left(x\right)\right)
 }\over{16}}+{{3\,x^8\,\left(y-\xi\left(x\right)\right)}\over{2}}\\
 +3\,x^6\,\left(y-\xi\left(x\right)\right)-4\,x^4\,\left(y-\xi\left(x
 \right)\right)+{{x^{12}}\over{64}}+{{3\,x^{10}}\over{16}}+{{3\,x^8
 }\over{4}}
\end{multline*}
The Newton polygon of $f$ with respect to $C$ is depicted in Figure \ref{fig:NP}.
It has three sides and four vertices, corresponding to the ``monomials''
$(y-\xi(x))^6$, $x^2(y-\xi(x))^2$, $x^4(y-\xi(x))$ and $x^8$.
The curve $V:f(x,y)=0$ has four branches through $O$, all smooth;
two of them are transverse
to $C$ and map to the first side of the Newton polygon; one of them is tangent
to $C$ with intersection multiplicity 2, and maps to the second side; the last one
is tangent to $C$ and has intersection multiplicity 4 with it, and maps to the
third side.
\begin{figure}
\centering
\begin{tikzpicture}[scale=.8]
    \tikzstyle{monomial}=[draw,circle,fill=black,minimum size=2.5pt,
                            inner sep=0pt]
  \fill[gray!50!white] (0,3.3) -- (0,3)  --
  (1,1)  --   (2,0.5) --  (4,0)   -- (6.3,0) -- (6.3,3.3);
  \draw[step=.5,help lines] (0,0) grid (6.3,3.3);
  \draw[thick] (0,3.3) -- (0,3) node [monomial,label=left:{${\frak v}_0$}] {}  --
  (1,1)  node [monomial,label=225:{${\frak v}_1$}] {}  --
  (2,0.5) node [monomial,label=240:{${\frak v}_2$}]{} --
  (4,0)  node [monomial,label=below:{${\frak v}_3$}] {}  -- (6.3,0) ;
  \draw (1,2.5) node [monomial] {};
  \draw (1,2) node [monomial] {};
  \draw (2,2) node [monomial] {};
  \draw (2,1.5) node [monomial] {};
  \draw (2,1) node [monomial] {};
  \draw (3,.5) node [monomial] {};
  \draw (3,1.5) node [monomial] {};
  \draw (3,1) node [monomial] {};
  \draw (4,1) node [monomial] {};
  \draw (4,.5) node [monomial] {};
  \draw (5,.5) node [monomial] {};
  \draw (5,0) node [monomial] {};
  \draw (6,0) node [monomial] {};
  \draw (4.25,2.25) node {${\overline \NP}(C,f)$};
\end{tikzpicture} \phantom{sp}
\begin{tikzpicture}[scale=0.8, auto]
 \draw[step=.5,help lines] (0,0) grid (4.8,4.3);
  \draw[thick] (0.5,2) to node [swap,fill=white] {\small{$2+2s$}} (1,3)
  to node [swap,fill=white] {\small{$4+s$}}  (2,4)
  to node [swap,fill=white] {\small 8} (4.8,4) ;
  \draw (4.8,0)--(0,0)--(0,4.3);
%
\end{tikzpicture}
\caption{\label{fig:NP}
\small{The Newton polygon of
Example \ref{exa:Newton}. Each dot represents a ``monomial'' in the
$C$-expansion of $f$; only four of them create vertices of the polygon.
At the right-hand side, the corresponding function $v_1(C,s,;f)$ for $s\geqslant 1$.
The three linear pieces correspond to the vertices $\frak v_1=(2,2)$,
$\frak v_2=(4,1), \frak v_3=(8,0)$,
as described in proposition \ref{lem:trop}.}}
\end{figure}
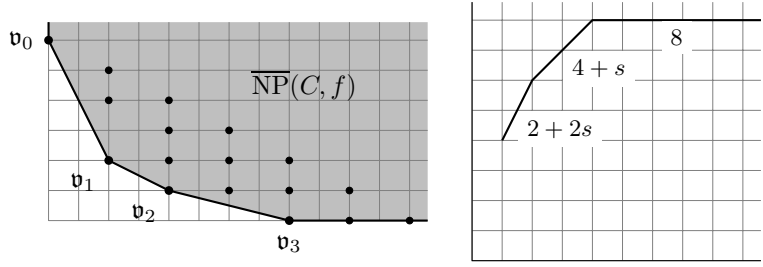
\end{Exa}

 \subsection{Quasimonomial valuations and the associated rank 2 valuations}\label{ssec:related}

We keep the above notation. As we saw in \S\ref {ssec:NP}, we have a finite sequence
$$s_0:=1<s_1<\ldots s_h<s_{h+1}:=+\infty$$ such that $v_1(C,s;f)$ is linear (hence differentiable) in each of the intervals $(s_k,s_{k+1})$,
for $k=0,\ldots, h$. The derivative in these intervals is constant and integral. At $s_k$, with $k=0,\ldots, h+1$, there are  the right and left derivatives of $v_1(C,s;f)$ (at $s_0=1$ (resp. at $s_{h+1}=+\infty$) there is only the right (resp. left) derivative). So we have:

\begin{Cor}
 For any curve $C$ smooth at $O$ and $f\in K(x,y)-\{0\}$ regular at $O$, the function
 $v_1(C;f)$ has everywhere in $(1,+\infty)$ (resp. in $[1,+\infty)$) left
 (resp.  right) derivative. We will denote them by $\partial_- v_1(C;f)$ (resp. $\partial_+ v_1(C;f)$).
\end{Cor}

\begin{Pro}\label{deriv_val}
 For any curve $C$ smooth at $O$, every $s\in \Q$, $s> 1$ and every $f\in K(x,y)-\{0\}$ set
 \[
v_- (C,s)(f)\ := \ (v_1(C,s;f),-\partial_- v_1(C;f)(s))
\]
\[
v_+ (C,s)(f)\ :=\ (v_1(C,s;f), \partial_+ v_1(C;f)(s)).
 \]
This defines two rank 2 valuations  $v_- (C,s)$ and
 $v_+ (C,s)$ with home at $O$. For $s=1$, the valuation $v_+ (C,s)$
defined as above is also a rank 2 valuation with home at $O$.
\end{Pro}
\begin{proof}
Let $f\in K[x,y]$ and let $(x,\xi(x))$ be a local parametrization of $C$. With notation as in \eqref {eq:c-exp}, then \eqref {monomialdef} holds, thus
$$\partial_- v_1(C,\,\cdot\,;f)(s)=\max\{j|\exists i: a_{ij}\ne 0\ ,i+sj=v_1(C,s;f)\}$$
and
\begin{equation}\label{eq:+} \partial_+ v_1(C,\,\cdot\,;f)(s)=\min\{j|\exists i: a_{ij}\ne 0\ ,i+sj=v_1(C,s;f)\}.\end{equation}
The fact that both $v_-(C,s), v_+(C,s): \mathbb K^* \longrightarrow \Q^2_{\lex}$ are valuations follows from basic properties of multiplication of Laurent series and $\min$ and is left to the reader.
Furthermore, if $f$ is regular at $O$ then
$v_1(C,s;f)>0$ if and only if $f(O)= 0$. This implies that $O$ is the home of $v_- (C,s)$ and $v_+ (C,s)$.

Obviously $v_- (C,s)$ and $v_+ (C,s)$ have rank at most 2.
We will show that they have rank greater than 1.
Let $f_0\in K[x,y]$ be such that $f_0=0$ is an equation of $C$ (this, for fixed $s$, is no restriction by Remark \ref {rem:ocs}). We have $v_\pm (C,s;f_0)=(s,\pm1)$ (see Example \ref {ex:C}). Moreover if $s=\frac pq $ for coprime positive integers $p,q$ and $f_1=\frac {f_0^q}{x^p}$ then $v_\pm (C,s;f_1)=(0,\pm q)$.
Thus for every positive integer $k$ we have
\[(0,0)<\pm k v_\pm (C,s;f_1)<v_\pm (C,s;x),\] which is impossible for a rank 1 valuation.
\end{proof}

\begin{Rem}
 For irrational $s$, the expressions $v_-$ and $v_+$ (as defined in  Proposition~\ref{deriv_val}),
 are valuations with home at $O$, but they are both equivalent to $v_1$ (and so have real rank 1 and rational rank 2).
 We will not need this fact, and we leave the proof to the interested reader.
\end{Rem}

\begin{Rem}\label{remark:volume-OK}
 Write $s=\frac pq$ with $p, q$ coprime positive integers. Then
the value group of $v_- (C,s)$ and $v_+ (C,s)$
is $(\Z\frac1q \times \Z)_{\lex}\subset \Q^2_{\lex} $.
In this case, we will denote by
 \[
 \Delta_{C,s_-},  \Delta_{C,s_+} \ \subseteq \R^2_+
 \]
 the \NO bodies associated to the line bundle $\O_{\P^2}(1)$ with respect to the valuation $v_- (C,s)$ and $v_+ (C,s)$ respectively.

Since $v_\pm(C,s)$ have maximal rank but their
value groups do not equal $\Z^2_{\lex}$, the volumes of \NO bodies
associated to these
valuations need not satisfy Theorem~\ref{thm:LM2}. However,
there are order preserving elements of $\operatorname{GL}(2,\Q)$
relating the $v_\pm$ valuations to valuations with values in $\Z^2_{\lex}$.
In \S\ref{sec:pm-flag} below we compute these lower triangular
matrices, which turn out to have determinant 1, and so preserve
the volume. Therefore Theorem~\ref{thm:LM2} also applies
to~$v_\pm (C,s)$, and $\vol  \Delta_{C,s_-}= \vol  \Delta_{C,s_+}
= (\vol (\O_{\P^2}(1)))/2=1/2$.
\end{Rem}

\subsection{The $\hat \mu$ invariant}\label{ssec:mu}

Let $v_1$ be a rank 1 valuation centred at a smooth point $x$ of a normal irreducible projective  surface $X$,
and let $D$ be a big Cartier divisor on~$X$.
Following \cite{DHKRS}, we set
\[
\mu_D (v_1) \deq  \max\{v_1(f)\,|\, f\in H^0(X,\O_X(D))-\{0\} \}\ , \
\text{and} \
\hat \mu_D (v_1) \deq  \lim_{k \to \infty}\frac{\mu_{kD}(v_1)}{k}\ .
\]
%
%
If $v=(v_1,v_2)$ is a valuation of rank 2 centred at $x$, then $\Delta_v(D)$
lies in the strip
\[
\{(t,u)\in \R^2 | 0\leqslant t\leqslant \hat\mu_D (v_1)\}\ ;
\]
and  its projection to the $t$-axis lies the interval
$[0,\hat\mu_D (v_1)]$, coinciding with it if and only if $x\not\in {\rm Neg}(D)$ (see Theorem \ref {thm:LK}).

In order to simplify notation, we will set
\[
\mu_D(C,s) = \mu_D(\vv{C}{s}), \,\,\, \text{
 and} \,\,\,
\hat\mu_D(C,s)=\hat \mu_D (\vv{C}{s}).
\]
If $X=\mathbb P^2$, $x=O$ and $D$ is a line, we  drop the subscript $D$
for $\hat\mu_D(C,s)$ and we write $\mu_d(C,s)$ instead of $\mu_{dD}(C,s)$ for any non-negative integer $d$.

From \cite{DHKRS} we know that  the function
$\hat\mu:\mathcal {QM}\rightarrow \R$ is
lower semicontinuous for the weak topology and continuous
for the strong topology, i.e.,  $\hat\mu(C,s)$ is continuous for  $s\in [1,+\infty)$ (see \cite [Proposition 3.9]{DHKRS}). Moreover
$\hat\mu(C,s)\geqslant \sqrt{s}$ \cite{DHKRS}. 
If $\hat\mu(C,s)=\sqrt{s}$, then $v_1(C,s)$ is said to be \emph{minimal} (the concept of \emph{minimal valuation} is more general, see \cite {DHKRS}, but we will not need it here).
We recall from \cite{DHKRS} the following:

\begin{Con} [{\cite[Conjecture 5.11]{DHKRS}}]\label{con:nn} If $C$ is \emph{sufficiently general} (in a sense which is made precise in l.c.) and $s\geqslant 8+ \frac 1{36}$, then $\hat\mu(C,s)=\sqrt{s}$.
\end{Con}

\begin{Rem}
According to \cite [Proposition 5.4]{DHKRS}, this Conjecture (actually a  weaker form of it, considering only $s\geqslant 9$ and $C$ any curve), implies Nagata's Conjecture.
\end{Rem}

\begin{Rem}\label{muhatDHKRS}
We  recall from \cite{DHKRS} some known values of $\hat\mu (C,s)$.
\begin{itemize}
 \item If $C$ is a line, then
 \[
 \hat\mu(C,s)=
\begin{cases}
 s & \text{if }1\leqslant s\leqslant 2 \\
 2 & \text{if }2 \leqslant s
\end{cases}
 \]
 \item If $C$ is a conic, then
 \[
 \hat\mu(C,s)=
\begin{cases}
 s & \text{if }1\leqslant s\leqslant 2 \\
 2 & \text{if } 2 \leqslant s\leqslant 4 \\
 \frac{s}{2} & \text{if } 4 \leqslant s\leqslant 5 \\
 5/2 & \text{if }5 \leqslant s
\end{cases}
 \]

\item If $s\leqslant 7+1/9$ and $\deg (C) \geqslant 3$, then
\[
\hat\mu(C,s)=
\begin{cases}
 \frac{F_{i-2}}{F_{i}}\,s & \text{if } \frac{F_{i}^2}{F_{i-2}^2} \leqslant s\leqslant \frac{F_{i+2}}{F_{i-2}},\,\,\,
  i\geqslant 1\text{ odd},\\
 \frac{F_{i+2}}{F_{i}} & \text{if }\frac{F_{i+2}}{F_{i-2}} \leqslant s\leqslant \frac{F_{i+2}^2}{F_{i}^2}\ ,
 \,\,\, i\geqslant 1\text{ odd},\\
 \frac{1+s}{3} & \text{if }\phi^4\leqslant s\leqslant 7 ,\\
 \frac{8}{3} & \text{if }7  \leqslant s\leqslant 7+\frac 19 ,
\end{cases}
\]
where $F_{-1}=1$, $F_0=0$ and $F_{i+1}=F_i+F_{i-1}$ are the \emph{Fibonacci numbers}, and $$\phi=\frac{1+\sqrt{5}}{2}=\lim\limits_{i\to \infty} \frac{F_{i+1}}{F_i}$$
is the  \emph{golden ratio}.

The values of $\hat \mu$ above are computed using the series of \emph{Orevkov rational cuspidal curves} (see \cite {Ore02} and Proposition \ref {prop:Or} below). There are a few more sporadic values of $s$ in the range $[7+1/9,9]$ where the value
of $\hat\mu$ is known, see \cite{DHKRS} for details.

\item If $s$ is an integer square and $C$ is a general curve of degree
at least $\sqrt{s}$, then one has $\hat\mu(C,s)=\sqrt{s}$.
\end{itemize}
\end{Rem}

\section{Cluster of centres and associated flags}
\label{flagmodel}


In this section the main goal is to introduce the geometric structures related to valuations $\vv{C}{s}$ and $v_{\pm}(C,s)$. We give a full description of how to find the birational model of $X$ (the cluster of centres together with their weights) on which these two valuations are equivalent to a flag valuation on this model.

\subsection{Weighted cluster of centres}

 As usual, we  will refer to the case
 $$x=O\in \mathbb A^2\subset \mathbb P^2:=X_0.$$

Each valuation $v$ with centre $O\in\P^2$ determines a \emph{cluster of centres} as follows. Let $p_1=\cent_{X_0}(v)=O$.
Consider the blowup $\pi_1:X_1\rightarrow X_0$ of $p_1$ and let $E_1\subset X_1$ be the corresponding exceptional divisor. Then $\cent_{X_1}(v)$ may either be  $E_1$ or a point $p_2\in E_1$. Iteratively blowing up the centres $p_1, p_2, \dots$ of $v$ we either end up, after $k\geqslant 1$ steps, with a surface $X_{k}$ dominating $\mathbb P^2$,  where the centre of $v$ is the exceptional divisor $E_{k}$. In this case $v$ is discrete of rank $1$, given by the order of vanishing along $E_k$, by Remark \ref {divisorial}. Otherwise, this process goes on indefinitely. In particular, for quasimonomial valuations $\vv{C}{s}$, the process terminates if and only if the characteristic
exponent $s$ is rational.

Let $v=(v_1,v_2)$ be now a rank $2$ valuation whose truncation $v_1$ is quasimonomial. From Abhyankar's inequalities, \cite [p. 12]{FJ04}, one concludes that $v_1$ has rational rank $1$. Hence, by Remark \ref {rem:qm}, we have $v_1=v_1(C,s)$ for
some $s\in \Q$. By above then, the sequence of centres of $v$ is infinite, whereas the sequence of $1$-dimensional homes (centres of $v_1$) terminates at a blowup $X_k$ where $\cent_{X_k}(v_1)=E_k$ is an exceptional divisor. In particular, $v$ is equivalent to the valuation $\nu_{Y_\bullet}$, defined by the flag
\[
Y_\bullet: X_k \supset E_k \supset \cent_{X_k}(v)=p_{k+1}.
\]
The punchline of all this is that the process of blowing up all $0$-dimensional centres of the truncation
provides an effective method to find a model  where a given rank 2 valuation becomes a flag valuation. By Theorem
\ref{flag_valuation}, such a model exists for every valuation of maximal
rank on a projective variety. The above method works for any valuation of rank 2 on any projective surface (i.e., not necessarily $\mathbb P^2$).

For each centre $p_i$ of a valuation $v$, general curves through $p_i$
and smooth at $p_i$
have the same value $e_i=v(E_i)$, which we call the \emph{weight} of $p_i$ for $v$.
Following \cite[Chapter 4]{Cas00}, we call the (possibly infinite) sequence
$\mathfrak K_v=(p_1^{e_1}, p_2^{e_2},\dots)$ the \emph{weighted cluster of centres} of $v$.
In general a sequence like $\mathfrak K=(p_1^{e_1}, p_2^{e_2},\dots)$ is called a \emph{weighted cluster of points}
and ${\rm supp}(\mathfrak K)=(p_1, p_2,\dots)$ is called its \emph{support}.

 If $v$ is a valuation with centre at $O$, then its weighted cluster
 of centres completely  determines $v$.
 Indeed,  for every effective divisor $Z$ on $\P^2$, one has
 \begin{equation}
\label{eq:proximity}
 v(Z)=\sum_i e_i\cdot \mult_{p_i} (\widetilde Z_i),
\end{equation}
 where $\widetilde Z_i$ is the
proper transform of $Z$ on  $X_i$, whenever the sum on the right
has finitely many non-zero terms.
This is always the case unless $v$ is a rank 2 valuation with
home at a curve through $O$ and $Z$ contains this curve;
in particular, for valuations
of rank 1, such as $v_1(C,s)$, formula
\eqref{eq:proximity} always computes $v(Z)$ \cite[\S8.2]{Cas00}.

As usual, with the above notation,  we say that a curve $Z$ passes through
an \emph{infinitely near point} $p_i\in X_i$ if its proper transform $\widetilde Z_i$ on $X_i$ contains $p_i$.

\subsection{The cluster associated to $v_1(C,s)$}
The description of the cluster $\mathfrak K_{(C,s)}:= \mathfrak K_{\vv{C}{s}}$ is classical and we refer for complete proofs to \cite{Cas00}. Here, we merely focus on the construction of the cluster of centres for $\vv{C}{s}$ and its main properties that will be used in the next section. The cluster $\mathfrak K_{(C,s)}$ is a very specific one, and we will need the following definition to make things more clear.
\begin{Def}\label{def:prox}
With notation as above, the centre $p_i\in X_{i-1}$ is  called \emph{proximate to $p_j\in X_j$}, for $1\leqslant j<i\leqslant k$,
(and one writes $p_i\succ p_j$) if $p_i$ belongs to the proper transform $E_{i-1,j}$ on $X_{i-1}$ of the exceptional
divisor $E_{j+1}:=E_{j+1,j}$ over $p_j\in X_{j-1}$. For the cluster $\mathfrak K_{(C,s)}$, each $p_i$, with $i\geqslant 2$, is proximate to $p_{i-1}$ and to at most one other centre $p_j$, with $1\leqslant j<i-1$; in this case
$p_i=E_{{i-1,j}}\cap E_{i-1}$ and $p_i$ is called a   \emph{satellite point}. A point which is not satellite is called \emph{free}.
\end{Def}

We know that the support of the cluster $\mathfrak K_{(C,s)}= \mathfrak K_{\vv{C}{s}}$ is determined by the continued fraction expansion
\[
s=\frac pq =[n_1;n_2,\ldots, n_r] = n_1+\frac{1}{n_2+\frac{1}{n_3+\frac{1}{\ddots}}}\ ,
\]
where $p,q$ are coprime and $r\in \mathbb{Z}_{>0}\cup \{\infty\}$. Before moving forward, let's fix some notation. Let $k_i=n_1+\dots+n_i$ and $k=k_r$. We denote by
$$s_i=\frac {p_i}{q_i}=[n_1;n_2,n_3,\ldots, n_i], \,\,\, \text{for}\,\,\, i=1,\ldots, r$$
the \emph{partial fractions} of $s$, where ${p_i},{q_i}$ are coprime positive integers.

First, the cluster $\mathfrak K_{(C,s)}$ consists of $k=\sum n_i$ centres (if $s$ is irrational there are infinitely many centres). We focus here only on the case when $s\in\Q$, hence $r<\infty$. Set $\mathfrak K=\mathfrak K_{(C,s)}$ and for each $i=0,\ldots,k-1$ let $\pi_i: X_{i+1}\to X_{i}$ be the blow-up of $X_{i}$ at the centre $p_{i+1}$ with exceptional divisor $E_{i+1}$. As usual we start with $X_0:=\mathbb P^2$. Denote $X_{\mathfrak K}:=X_k$ and let $\pi: X_\mathfrak R\to X$ be the composition of the $k$ blowups.

With this in hand, we explain the algorithm for the construction of $\mathfrak K$. If $s=n_1$ (so that $r=1$), then the centre $p_{i+1}$ is the point of intersection of the proper transform of $C$ through the map $X_{i}\rightarrow X_0$ and the exceptional divisor $E_{i}$ of $\pi_{i-1}$, for each $i=1, \ldots , n_1-1$. When $r>1$, then the first $n_1+1$ (including $p_1$) centres of $\mathfrak K$ are obtained as in the case when $s$ was integral, i.e. these points are chosen to be free. The rest are satellites: starting from $p_{{n_1}+1}$ there are $n_2+1$ points proximate to~$p_{n_1}$, i.e. each $p_j$ is the point of intersection of the proper transform of $E_{n_1}$ and the exceptional divisor $E_{j-1}$. Thus, $E_{n_1}$ plays the same role for these centres as $C$ did in the first step. Then, one chooses $n_3+1$ points proximate to $p_{n_1+n_2}$ and so on. Since $r<\infty$, then the last $n_r$ points (not $n_r+1$) are proximate to $p_{n_1+\dots+n_{r-1}}$. The final space $X_{\mathfrak K}$ is where $\vv{C}{s}$ becomes a divisorial valuation, defined by the order of vanishing along the exceptional divisor $E_{k}\subseteq X_{\mathfrak K}$. Finally note that $C$ plays a role only in the choice of the first $n_1$ centres. This is due to Remark~\ref{rem:ocs}, saying that the valuation $\vv{C}{s}$ depends only on the $\lfloor s \rfloor$-th jet of $C$.

The weights in $\mathfrak K_{(C,s)}$ are proportional to the multiplicities of the curve with Puiseux series $y=\xi(x)+\theta x^ s$ at the points of ${\rm supp}(\mathfrak K_{(C,s)})$. These and the continued fraction expansion are computed as follows. Consider the euclidean divisions
\[
m_i=n_{i+1}m_{i+1}+m_{i+2}\quad {\rm of}\quad m_i\quad {\rm by}\quad m_{i+1},\quad \text{for}\quad  i=0,\ldots,r-1 \ ,
\]
where $m_0:= p, m_1:= q$. Then the first $n_1$ points of $\mathfrak K_{(C,s)}$ have weight
$$e_1=e_2=\dots=e_{n_1}=\frac{m_1}{q}=1,$$
the subsequent $n_2$ points have weight $m_2/q$, ..., the final $n_r$ points have weight $m_r/q=1/q$. Therefore the  \emph{proximity equality}
\begin{equation}
  \label{eq:proximity-weights}
e_j \,=\, \sum_{p_i \succ p_j} e_i \
\end{equation}
holds for all $j=0,\ldots, k-1$. Conversely, for every weighted cluster $\mathfrak K$ with finite support, in which every point is infinitely near to the previous one,  no satellite point precedes a free point, and the proximity equality holds, there exist
a smooth curve through $O$ and a rational number $s$ such that $\mathfrak K=\mathfrak K_{(C,s)}$.

\subsection{$v_\pm(C,s)$ and the associated flag valuation}
\label{sec:pm-flag}
In order to describe the flag valuation associated to $v_\pm(C,s)$, it is necessary to understand first the intersection theory of all the proper and total transforms of the exceptional divisors on $X_{\mathfrak K}$.

To ease notation, let $A_i$ (resp. $B_i$) be the proper (resp. total) transform of $E_i\subset X_i$ on $X_\mathfrak K$, for $i\in \{0,\ldots, k-1\}$.
Then:
\begin{Lem}\label{lem:intersections}
\begin{inparaenum}
\item [(i)] $A_k=E_k$ is the only curve with $A_i^2=-1$ for any $i=1, \ldots, k$; \\
\item [(ii)] $A_{k_i}=B_{k_i}-B_{k_i+1}-\dots-B_{k_{i+1}+1}$ and $A_{k_i}^2=-2-n_{i+1};$ for each $1\leqslant i<r-1$;\\
\item [(iii)] $A_{k_{r-1}}=B_{k_{r-1}}-B_{k_{r-1}+1}-\dots-B_{k}$ and $A_{k_{r-1}}^2=-1-n_r$;\\
\item [(iv)] $A_j=B_j-B_{j+1}$ and $A_j^2=-2$ for every $ j\in \{1,\ldots ,k\}\setminus \{k_1,\ldots,k_r\}$.
\end{inparaenum}
\end{Lem}

The sheaf $\pi^*(\mathcal I_{p_1\vert \mathbb P^ 2})$ is invertible on $X_\mathfrak K$ and defines the \emph{fundamental cycle} $E$ of~$\pi$. Write $E= \sum_{i=1}^k a_iA_i$. Then, making use of Lemma~\ref{lem:intersections}, the multiplicities $a_i$ can be easily computed as follows:
\begin{Lem}\label{lem:fundamental}
If one assumes $k_0=a_0=0$ and $a_1=1$, then the multiplicities of the fundamental cycle $E$ are computed by the following formula
\[
a_i=a_{k_{j-1}}(i-k_{j-1}-1)+a_{k_{j-1}+1}, \textup{ for }k_{j-1}+2\leqslant i\leqslant  k_{j}+\epsilon\textup{ and } 1\leqslant j \leqslant r
\]
where $\epsilon=0$ if $j=r$ and $\epsilon=1$ otherwise. In particular, $a_{k_j+1}=a_{k_j}+a_{k_{j-1}}$ for $1\leqslant j\leqslant r-1$.
\end{Lem}
\begin{Rem}\label{rem:multiplicites}
Using both the equalities in Lemma~\ref{lem:fundamental}, one has $a_{k_j}=n_ja_{k_{j-1}}+a_{k_{j-2}}$ for any $j=2,\ldots ,r$, where $a_0=0$ and $a_{k_1}=1$.

On the other hand, using the partial fractions $s_i=\frac {p_i}{q_i}$ of $s=\frac{p}{q}$, one has the same recursive relations $q_j=n_jq_{j-1}+q_{j-2}$ for $j=2,\ldots ,r$, with $q_0=0$ and $q_1=1$. Thus, we get that $a_{k_j}=q_j$ for any $j=0, \ldots , r$. In particular, we have $a_{k_r}=q$.
\end{Rem}
In the following the pair $(p_{r-1},q_{r-1})$ of the partial fraction $s_{r-1}=\frac{p_{r-1}}{q_{r-1}}$ will play an important role, so we fix some notation. When $s$ is not an integer (i.e. $r\geqslant 2$), we set
$$p'=p_{r-1}, \quad q'=q_{r-1} \,\,\, \text{so that}\,\,\, s_{r-1}=\frac {p'}{q'}$$
If $s$ is an integer, i.e., $r=1$, then we set $p'=q'=1$.

In order to find the flags on $X_{\mathfrak K}$ associated to $v_\pm(C,s)$, we need to have a better understanding of the cycle $E$ through its dual graph. The dual graph of $E$ is a \emph{chain}, i.e. a tree with only two end points, corresponding to $A_1$ and $A_{n_1+1}$. If  $A$ is the proper transform of $C$ on $X_\mathfrak K$, then $A$ intersects $E$ only at one point on $A_{n_1+1}$.  Thus the dual graph of  $A+E$ is also a chain, with end points corresponding to $A_1$ and $A$. The curve $A_k$  intersects exactly two other components of $A+E$, precisely:\\
\begin{enumerate}
\item [(a)] if $s$ is not an integer (so that $r\geqslant 2$), then $A_k$ intersects
$A_{k-1}$ and $A_{k_{r-1}}$, whose multiplicities in the cycle $A+E$ are $a_{k-1}=q-q'$ and $a_{k_{r-1}}=q'$;\\
\item [(b)] if $s$ is an integer (so that $s=k=n_1$), then $A_k$  intersects $A_{k-1}$ and $A$, both having multiplicity one.
\end{enumerate}

Note that $A+E-A_k$ has two connected components, only one containing $A$. We denote this component by $A_+$ and the other by $A_-$.
We will denote by $x_\pm$ the intersection point of $A_k$ with $A_\pm$, and by $x$ the general point of $A_k$.

The  total transform $C^ *$ on $X_\mathfrak K$ of $C$ has the same support as $A+E$, but the multiplicities are different. In particular, denoting
\[
p''=
\begin{cases}
  p' & \text{if }r \text{ is odd} \\
 p-p' & \text{if }r \text{ is even}
\end{cases} \qquad
q''=
\begin{cases}
  q' & \text{if }r \text{ is odd} \\
 q-q' & \text{if }r \text{ is even}
\end{cases} \qquad
\]

\begin{Lem}\label{lem:loc} \begin{inparaenum}
\item [(i)] The divisor $C^ *$  contains $A_k$ with multiplicity $p$ and $C^ *-pA_k$ passes through $x_+$
(resp. $x_-$) with
multiplicity $p''$(resp. $p-p''$).\\
\item [(ii)]  The total transform $L$ of the line $x=0$ on $X_\mathfrak K$ contains $A_k$ with multiplicity $q$ and $L-qA_k$ passes through $x_+$ (resp. $x_-$) with
multiplicity $q''$ (resp. $q-q''$).
\end{inparaenum}
\end{Lem}
\begin{proof} We prove only (i), the proof of (ii) being analogous.

When $s$ is an integer the assertion is trivial. So, assume that $s$ is not an integer (i.e. $r\geqslant 2$). We first show that the multiplicity of $A_k$ in $C^*$ is equal to $p$. This is done inductively on $k=n_1+\ldots +n_r$. From the standard properties of continuous fractions it is worth to note that the numerator of $[n_1;n_2,\ldots ,n_r-1]$ is equal to $p-p_{r-1}$, where $p_{r-1}$ is the numerator of the continued fraction
\[
s_{r-1} \ = \frac{p_{r-1}}{q_{r-1}} \ = \ [n_1;n_2,\ldots ,n_{r-1}] \ .
\]
The multiplicity of $A_k$ in $C^ *$ is the same as the multiplicity of $A_k$ in $B_1+\ldots +B_{n_1+1}$. So, using Lemma~\ref{lem:intersections} repeatedly, the statement follows easily.

The multipliticies of $C^ *-pA_k$ at $x_+$ and $x_-$ equal the multiplicities in $C^ *$ of $A_{k_{r-1}}$ and of $A_{k-1}$ respectively in this order if $r$ is odd, and reversed if $r$ is even (as $r\geqslant 2$). Arguing as before, one deduces easily also these statements.\end{proof}

\begin{Exa} \label{616}
Consider $s=48/7$. Its continued fraction is $[6;1,6]=6+\frac{1}{1+1/6}$.
Therefore the cluster of centers of $v_1(C,s)$ consists of 7 free
points on $C$ followed by six
satellites; of these, $p_8$ is proximate to $p_6$ and $p_7$, and each of
$p_9, \dots, p_{13}$ is proximate to its predecessor and to $p_7$.
See Figure \ref{fig:Enriquesexample}, where the weights $e_i$
are printed in boldface: $e_1=\mathbf{1}$, $e_2=\mathbf{6/7}$,
$e_3=\mathbf{1/7}$.

The proximities mean that the exceptional components are $A_6=B_6-(B_7+B_8)$,
$A_7=B_7-(B_8+\dots+B_{13})$, $A_{13}=B_{13}$ and, for all
$i\ne 6,7,13$, $A_i=B_i-B_{i+1}$. Solving for $B_1=E$ one gets the fundamental cycle
\[  E=A_1+A_2+A_3+A_4+A_5+A_6+A_7+2A_8+3A_9
+4A_{10}+5A_{11}+6A_{12}+7A_{13}. \]
Since $C$ goes through $p_1, \dots, p_7$ with multiplicity 1, its pullback to
$X_{\frak K}$ is
\begin{multline*}
 C^*=\tilde C+B_1+\dots+B_7=\\
 \tilde C + A_1+2A_2+3A_3+4A_4+5A_5+6A_6+7A_7+\\
 13A_8+20A_9+27A_{10}+34A_{11}+41A_{12}+48A_{13}.
\end{multline*}
Clusters are often represented by means of \emph{Enriques diagrams}
(see \cite[p.98]{Cas00}) as explained in Figure \ref{fig:Enriquesexample}
illustrating this example.
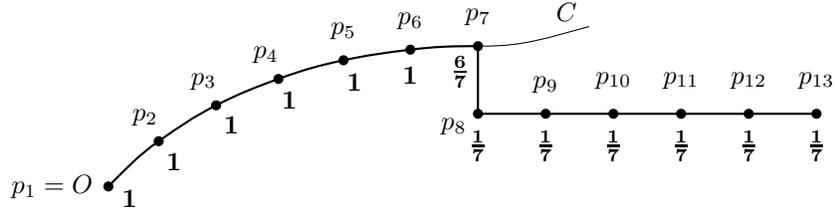
\begin{figure}
\centering
 \begin{tikzpicture}[label distance=3pt,auto]
    \tikzstyle{every node}=[draw,circle,fill=black,minimum size=3pt,
                            inner sep=0pt]
\draw [thick] (0,0) node (O) [label=left:{$p_1=O$},label=345:{\bf 1}] {} to [out=45,in=215]
     ++(42:.9cm) node (p2) [label=100:$p_2$,label=300:{\bf 1}] {} to [out=35, in=207]
     ++(32:.9cm) node (p3) [label=95:$p_3$,label=300:{\bf 1}] {} to [out=27, in=200]
     ++(23:.9cm) node (p4) [label=95:$p_4$,label=280:{\bf 1}] {} to [out=20, in=192]
     ++(16:.9cm) node (p5) [label=above:$p_5$,label=280:{\bf 1}] {} to [out=12, in=187]
     ++(9:.9cm) node (p6) [label=above:$p_6$,label=270:{\bf 1}] {} to [out=7, in=180]
     ++(3:.9cm) node (p7) [label=above:$p_7$,label=250:{$\mathbf{\frac67}$}] {} to
     ++(270:.9cm) node (p8) [label=185:$p_8$,label=270:{$\mathbf{\frac17}$}] {} to
     ++(0:.9cm) node (p9) [label=above:$p_9$,label=270:{$\mathbf{\frac17}$}] {} to
     ++(0:.9cm) node (p10) [label=above:$p_{10}$,label=270:{$\mathbf{\frac17}$}] {} to
     ++(0:.9cm) node (p11) [label=above:$p_{11}$,label=270:{$\mathbf{\frac17}$}] {} to
     ++(0:.9cm) node (p12) [label=above:$p_{12}$,label=270:{$\mathbf{\frac17}$}] {} to
     ++(0:.9cm) node (p13) [label=above:$p_{13}$,label=270:{$\mathbf{\frac17}$}] {};
\draw (p7) to  [out=0,in=195]   ++(10:1.5cm)  node
	(C) [draw=none,fill=none,label=150:$\tilde C$] {};
\end{tikzpicture}
\caption{\label{fig:Enriquesexample}
\small{The Enriques diagram \cite[3.9]{Cas00} of the cluster of centers of
Example \ref{616}. Each vertex in the diagram corresponds
to one of the points, with each vertex joined to its immediate prececessor by an edge;
edges are curved for free points, and straight segments for satellites, to
represent the rigidity of their position. The segments
joining a sequence of satellites proximate to the same point lie on the same line,
orthogonal to the immediately preceding edge.}}
\end{figure}
\end{Exa}

\begin{Pro}\label{prop:pm} In the above setting, the flags associated to the rank 2
 valuations $v_-(C,s)$ and $v_+(C,s)$ are
 \[
 Y_-:\quad X_\mathfrak K\supset A_k\supset x_-\quad \text{and}\quad
 Y_+:\quad X_\mathfrak K\supset A_k\supset x_+
 \]
 respectively.
\end{Pro}

\begin{proof} The above discussion makes it clear that $A_k$ is the centre of $v_1(C,s)$. It remains to prove that $x_\pm$ are the centres of $v_\pm(C,s)$. Let $\eta=0$ be a local equation of $A_k$ on $X_\mathfrak K$ around $x_+$. Consider  $f_1=f_0^q/x^p$ as in the proof of Proposition \ref{deriv_val}. By Lemma \ref {lem:loc}, the pull-back of $f_1$ to $X_\mathfrak K$ is not divisible by $\eta$. Again by Lemma~\ref{lem:loc}, it vanishes at $x_+$ with multiplicity $p''$.  Furthermore, by Proposition \ref{deriv_val}, one has $v_+(C,s)(f_1)>0$. By the same token, $f_1^{-1}$ is not divisible by $\eta$,
it vanishes at $x_-$ and has $v_-(C,s)(f_1^{-1})>0$, proving the assertion.\end{proof}

\begin{Rem}\label{rem:flag2} Unless $s$ is an integer and the sign $+$ holds, the valuations $v_\pm(C,s)$ are not equal to the evalutations associated to the flags $Y_\pm$ (see Remark \ref {rem:flag}), but they are equivalent to them.

Let $f_0=0$ be an equation of $C$ (which we may assume to be algebraic, see the proof of Proposition \ref{deriv_val}) and note that $K[[x,y]]\cong K[[x,f_0]]$. One has
\begin{equation}\label{eq:bat}
\quad v_\pm(C,s;x)=(1,0), \quad v_\pm(C,s;f_0)=(s,\pm1) \ ,
\end{equation}
by the proof of Proposition \ref{deriv_val}. By Lemma \ref {lem:loc}, one has
\[
\nu_{Y_+}(x)=(q,q''), \quad  \nu_{Y_+}(f_0)=(p,p'')
\]
\[
 \nu_{Y_-}(x)=(q,q-q''), \quad  \nu_{Y_-}(f_0)=(p,p-p'')\ .
\]
By standard properties of continued fractions, one has $pq'-qp'=(-1)^ r$. Thus
\[
v_+(C,s)= \begin{pmatrix}
 \frac 1 q & 0 \\ -q'' &  q
\end{pmatrix} \nu_{Y_+},
\]
\[v_-(C,s)= \begin{pmatrix}
 \frac 1 q & 0 \\ q''-q &  q
\end{pmatrix} \nu_{Y_-}.
\]
\end{Rem}
\begin{Rem}\label{rem:flag3}
The same relations, given in Remark \ref{rem:flag2}, hold for the corresponding \NO bodies. It is worth to note that both $2\times 2$ matrices transform vertical line into vertical lines. Furthermore, any vertical segment in $\Delta_{Y_{\pm}}(D)$ is translated into a vertical segment in $\Delta_{C,s_{\pm}}$ whose length is multiplied by a factor of $q$ with respect to the initial one, where $D$ is the class of a line.
\end{Rem}

\subsection{Zariski decomposition of valuative divisors}

In this subsection we will describe, with few details, some of the properties of the valuation $v_1$ that will be used in the next section. As before let $s=p/q\geqslant 1$ be a rational number and $\mathfrak K$ the cluster of centers associated to the rank $1$ valuation $v_1(C,s)$, with $\pi:X_{\mathfrak K}\rightarrow \P^2$ the sequence of blow-ups constructed in the previous section where the valuation $v_1(C,s)$ becomes equivalent to a valuation given by the order of vanishing along an exceptional curve on $X_{\mathfrak K}$. We will denote by
\[
B_s \ \deq  \ e_1B_1 + \ \ldots \ + e_kB_k \ ,
\]
where as usual $B_i$ is the pull-back of the exceptional divisor $E_i$ on $X_{\mathfrak K}$ and $e_i$ is the weight of the center $p_i$, whose blow-up is the curve $E_i$ (whereas $A_i$ is the strict transform in $X_{\mathfrak K}$ of $E_i$).
Note that the proximity equalities \eqref{eq:proximity-weights}
mean that $B_s\cdot A_i=0$ for all $1\leqslant i\leqslant k-1$, and that the weights
are also determined by these equalities and $e_k=1/q$
(see the proof of Theorem~2.2.2 in \cite{Cas00}).
Knowing this divisor $B_s$ we usually know almost everything about the valuation $v_1$. Using $(\ref{eq:proximity})$ one deduces the following:
\begin{Lem}\label{lem:intersection}
For a divisor $Z$ on $X_{\mathfrak K}$ not containing any of the exceptional curves $A_i$, one has $v_1(C,s;\pi_{*}(Z))=B_s\cdot Z$.
\end{Lem}
For the computation of \NO bodies, the following properties of~$B_s$ will also be
useful.
\begin{Lem}\label{lem:Bproperties}
 \begin{inparaenum}
\item[(i)] $B_s^2=-s$, $\textup{ord}_{A_k}(B_s)=p$, $(B_s\cdot A_i)=0$ for any $i=1,\ldots, k-1$, and $(B_s\cdot A_k)=-1/q$. \\
\item[(ii)] For every positive $r\in \Q$ such that the $\Q$-divisor
$D_r=D-rB_k$ is effective, the Zariski decomposition of
$D_r$ contains $\frac{r}{p}B_s-rB_k$ in its negative part.
\end{inparaenum}
\end{Lem}
\begin{proof}
The proof of~(i) is done inductively using the description of the cluster of centers obtained previously, in a similar way as the proof of Lemma~\ref{lem:fundamental}.

Let us prove~(ii).
If $k=1$ then $s=p=1$, $B_s=B_k=B_1$ and there is nothing to prove; so assume
$k>1$. Since the intersection matrix
of the collection $\{A_1, \dots , A_k\}$ is negative definite,
there exists a unique effective $\Q$-divisor
$N_\pi= \sum \nu_i A_i$ with
\begin{enumerate}
\item[(a)] $(D_r-N_\pi)\cdot A_i\geqslant 0$ for all $1\leqslant i\leqslant k$,
\item[(b)] $(D_r-N_\pi)\cdot A_i= 0$ for all $i$ with $n_i\ne 0$.
\end{enumerate}
The Zariski decomposition of $D_r$ \emph{relative to $\pi$}
is $D_r=P_\pi + N_\pi$ (see \cite[\S8]{CutSri}).
It satisfies
$H^0(X_{\mathfrak K},\mathcal{O}_{X_{\mathfrak K}}(mD_r)) \cong
H^0(X_{\mathfrak K},\mathcal{O}_{X_{\mathfrak K}}(mP_\pi))$ for all
$m$ such that $mD_r$ is a Weil divisor, and the negative
part of this relative Zariski decomposition is a part of
the full Zariski decomposition: $N_\pi \leqslant N$.

We claim that $N_\pi=\frac{r}{p}B_s-rB_k$.
Set $Z_0=D_r$. It is clear that $Z_0 \cdot A_{k-1}=-r<0$, and therefore
\[N_\pi \geqslant N_1:=\frac{r}{-A_{k-1}^2} A_{k-1}, \]
i.e., $\nu_{k-1}\geqslant -r/(A_{k-1}^2)$.
Define $Z_1=Z_0-N_1$. Then $Z_1 \cdot A_{k-2}=r/(A_{k-1}^2)<0$
so
\[N-N_1 \geqslant N_2:=\frac{r}{A_{k-1}^2A_{k-2}^2} A_{k-2}, \]
i.e., $\nu_{k-2}\geqslant r/(A_{k-1}^2A_{k-2}^2)$.
Define $Z_2=Z_1-N_2$. Iterating the process,
we see that $n_i>0$ for all $1\leqslant i\leqslant k-1$.
Using property $(b)$ together with the obvious equalities
$D_r\cdot A_i=0$ for $1\leqslant i \leqslant k-1$ and $D_r\cdot A_{k-1}=-rB_k\cdot A_{k-1}$,
we see that $(N_\pi + rB_k)\cdot A_i=0$ for $1\leqslant i\leqslant k-1$.
Since $N_\pi + rB_k$ is supported on the union of~$A_i$, this implies
$N_\pi + rB_k= m B_s$ for some real number $m$. Since the order of $B_k=A_k$ in
this divisor is $r$, this number must be $m=r/p$, and the claim follows.
\end{proof}

\section{\NO bodies on the tree $\mathcal {QM}$}\label {sec:tt}

From now on we will mainly concentrate on the study of $\Delta_{C,s_+}$ when $s$
varies in~\mbox{$[1,+\infty)$}. The case of $\Delta_{C,s_-}$ is not conceptually different and will be often left to the reader.

\subsection{General facts}\label{upper_border}\label{ss:gen}

\begin{Cor}\label{stat}  Let $C\subseteq \P^2$ be a curve of degree $d$. For any $s\geqslant 1$, one has the following inclusions
\[
\Delta_{1,\frac sd, \pm \frac 1d}\subseteq \Delta_{C,s_\pm} \subseteq \Delta_{\hat\mu,\hat\mu,\pm\frac {\hat\mu}s} \ ,
\]
where $\hat\mu=\hat\mu(C,s)$. Equality for the first inclusion holds if and only if $d=1$. Equality for the second one takes place if and only if $\hat\mu(C,s)= \sqrt s$.

 \end{Cor}
\begin{proof}
For the first inclusion, note that by the proof of Proposition \ref{deriv_val} evaluating an equation of $C$ and the variable $x$, forces both points $(1,0)$ and $(\frac sd,\pm\frac 1d)$ to be contained in $\Delta_{C,s_\pm}$. The origin is also contained in $\Delta_{C,s_\pm}$ since it is the valuation of any line not passing through the centre of the valuation. For the equality statement one uses that the area $\Delta_{C,s_{\pm}}$ is $\frac{1}{2}$, by Theorem \ref{thm:LM2} and Remark~\ref{remark:volume-OK}.

For the second inclusion notice first that by definition of $\hat\mu(C,s)$ from \S\ref {ssec:mu} one has that the convex sets $\Delta_{C,s_{\pm}}$ sit to the left of the vertical line $t=\hat\mu$.  To prove that $\Delta_{C,s_{+}}$
 also lies above the $t$-axis and below the line $t=su$, we need to show
\[
v_1(C,s;f)\geqslant s \cdot\partial_+ v_1(C,s;f)\geqslant 0  \ , \forall f\in K[x,y]\setminus \{0\} \ .
\]
Assuming \eqref {eq:c-exp} holds, this follows from \eqref {monomialdef} and \eqref {eq:+}, as $i+sj \geqslant sj$. The equality statement is again implied by the fact that the area of $\Delta_{C,s_{+}}$ is equal to $\frac{1}{2}$. The analogous
facts for $\Delta_{C,s_{-}}$ are left to the reader.
\end{proof}
\begin{Rem}
As a consequence of the above, then $\Delta_{C,s_+}$ sits above the $t$ axis and below the line with equation $su=t$ in the $(t,u)$ plane. Also, notice that $(0,0)$ and $(\frac {s}{d}, \frac {1}{d})$ are valuative points, where the latter is given by the valuation of a local equation of $C$ by Remark \eqref {rem:val}. Thus, every point with rational coordinates on the line $su=t$, lying between the origin and the point $(s/d,1/d)$, is valuative. The corresponding picture also holds for $\Delta_{C,s_-}$.
\end{Rem}

\begin{Rem} The valuation $v_{\rm gen}$ associated to the \emph{generic flag}
$$Y_{\rm gen}:\quad X_\mathfrak K\supset A_k\supset x$$
has nothing to do with $C$. On $X_\mathfrak K$ there is a smooth curve $\Gamma$ transversally intersecting $A_k$ at $x$. Its image on $X$ has local equation $\phi=0$ at $p_1$. Assume $X=\mathbb P^2$ and~\mbox{$\deg(\phi)=d$}. Then for $f=0$ a general line through $p_1$ one has $v_{\rm gen}(f)=(q,0)$ and
$v_{\rm gen}(\phi)=(\frac q d,\frac 1 d)$. Thus $\Delta_{\nu_{Y_{\rm gen}}}(D)$, with $D\in |\mathcal O_{\mathbb P^2}(1)|$, contains $\Delta_{q, \frac q d,\frac 1 d}$. Since in general $d>q$ (equality may hold only if $s=n_1$), then $\Delta_{\nu_{Y_{\rm gen}}}(D)$ is strictly larger than this triangle by Theorem \ref {thm:LM2}.
\end{Rem}

\begin{Rem}\label{rem:sh2}  By Corollary \ref{upper_border}, we see that Conjecture \ref{con:nn} is equivalent to asking whether for all $s\geqslant 8+\frac 1 {36}$ and $C$ general enough, one has
\[
\Delta_{C,s_+}=\Delta_{\sqrt s,\sqrt s,\frac 1 {\sqrt s}} \ .
\]
In particular, this implies Nagata's Conjecture and it shows how difficult it is to compute Newton--Okounkov bodies.
\end{Rem}

\begin{Cor}
  Let $C$ be a plane curve of degree $d$. Then $\Delta_{C,d^2_{\pm}}=\Delta_{d,d, \pm\frac 1d}$.
\end{Cor}

\begin{proof}
 The cluster of centres of $\vv{C}{d^ 2}$ consists of $p_1=O$ and the next $d^2-1$ points
 on $C$ infinitely near to $O$, i.e., $p_i=E_{i-1}\cap \tilde C$ for $i=2,\ldots,d^ 2$. If $A$ is the strict transform of $C$ on $X_{d^2}$, then
 $$A\equiv dD-E_1-\dots-E_{d^2}=dD-\sum_{i=1}^ {d^ 2} iA_i$$
 where $D$ is the pull back to $X_{d^2}$ of a line. Let $Z:=D-dA_{d^2}$, which can be written
 \[
Z \ \equiv \ \frac A d+ \sum_{i=1}^ {d^ 2-1} \frac i d A_i \ .
\]
Remark that this is actually the Zariski decomposition of $Z$, because $A$ is nef, as $A$ is irreducible and $A^2=0$, and $\sum_{i=1}^ {d^ 2-1} \frac i d A_i$  has clearly a negative definite intersection form. Also, $Z$ sits on the boundary of the pseudo-effective cone, as $A^2=0$. Thus, $Z-tA_{d^2}$ is not pseudo-effective for $t>0$.

Now the proof follows easily using Theorem~\ref{thm:LM}. Alternatively,
by Remark~\ref{rem:sh2}, it suffices to prove that $\hat\mu(C,d^2)={d}$.
Since $v_1=\textup{ord}_{E_{d^2}}$ as valuations
(by Remark~\ref{rem:flag2}, noting that $s$ is an integer)
one gets from the above paragraph that~\mbox{$\hat\mu(C,d^2)\leqslant {d}$}.
The opposite inequality follows from Lemma \ref{upper_border}.
 \end{proof}

\begin{Cor}\label{vertex_approach}
Let $C$ be a plane curve of degree $d$.
 For every $\epsilon>0$, there exists a non-zero $f_\epsilon\in K[x,y]$ whose
 $C$-expansion $f_\epsilon(x,y) = \sum a_{ij}x^i (y-\xi(x))^j$ satisfies:\\
\begin{inparaenum}
 \item [(i)] $\vv[f_\epsilon]{C}{d^2}=\min\{i+d^2 j|a_{ij}\ne 0\}\geqslant
 \deg(f_\epsilon)\cdot (d-\epsilon)$,\\
 \item [(ii)] $\partial_+\vv[f_\epsilon]{C}{d^2}=\min\{j|\exists i: a_{ij}\ne 0\ ,i+d^2j=v_1(C,s;f)\}
 \leqslant \deg(f_\epsilon) \cdot \epsilon$. \\
\end{inparaenum}
A similar statement holds for $\partial_-\vv[f]{C}{d^ 2}$.
\end{Cor}
By Corollary \ref{stat}, there exists a real number $\lambda >0$ such that
\begin{equation}\label{eq:inc}
\Delta_{\lambda, \lambda,\pm\frac \lambda s}\subset\Delta_{C,s_\pm}.
\end{equation}
This can be seen as an infinitesimal counter-part of Theorem \ref {thm:LK} for $X=\P^2$. When $s=1$ and $X$ is any smooth projective surface, these ideas were also developed in \cite{KL} along with Theorem \ref {thm:LK}. The largest $\lambda$ turned out to be the Seshadri constant of the divisor. This connection can be seen clearly in the following proposition, where the notation comes from \S4.4.
\begin{Pro}\label{prop:nef}
  Let $C\subset\P^2$ be a curve and $s=p/q\geqslant 1$. Let $\alpha\in \Q$ be
such that the $\Q$-divisor $\alpha D-B_s$ is nef. Then
$\Delta_{\frac{s}{\alpha},\frac{s}{\alpha},\frac{1}{\alpha}}\subseteq \Delta_{C,s_+}$.
\end{Pro}

\begin{proof}
Let's check first that $(\frac{s}{\alpha},0)\in\Delta_{C,s_+}$. By Remark~\ref{rem:flag2}, this is equivalent to showing that $(\frac{p}{\alpha},0)\in\Delta_{Y_+}(D)$. Since $\alpha D-B_s$ is nef, then there exists a sequence of effective ample divisors $H_n, n\geqslant 1$, where $x_+\notin \textup{Supp}(H_n)$, so that $D$ is the limit of $\frac{1}{\alpha}B_s+H_n$. So, the point $(\frac{p}{\alpha},0)$ is contained in $\Delta_{Y_+}(D)$, as $\textup{ord}_{A_k}(B_s)= p$.

By Remark~\ref{rem:flag3}, it remains to show that the height of the slice of $\Delta_{Y_{+}}(D)$ with first coordinate $t=\frac{p}{\alpha}$ is equal to $\frac{1}{q\alpha}$. For this, we apply Theorem \ref{thm:LM} for $t=\frac{p}{\alpha}$. Let $N_t+P_t$ be the Zariski decomposition of $D-tB_k$. By Lemma~\ref{lem:Bproperties}, we know that
\[
N_t \ - \ (1/\alpha)B_s+(p/\alpha)B_k \textup{ is effective.}
\]
Thus, one has $P_t\leqslant D-(1/\alpha)B_s$; but the latter $\Q$-divisor is nef by hypothesis, therefore $N_t=(1/\alpha)(B_s-pB_k)$
and $P_t= D-(1/\alpha)B_s$. In particular, the height of $\Delta_{Y_+}$ at $t=\frac{p}{\alpha}$ is equal to $(1/\alpha)P_t\cdot B_k=(1/\alpha)e_k=1/(q\alpha)$.
\end{proof}
Based on the previous statement, it is natural to introduce the following constant
\[
\lambda (C,s) \ \deq \ \textup{max}\{\lambda > 0 \ | \ \Delta_{\lambda, \lambda,\pm\frac \lambda s}\subset\Delta_{C,s_+}\} \ .
\]
As mentioned before, when $s=1$, the constant $\lambda(C,s)$ is nothing else than the Seshadri constant of $D$, the class of a line, at the origin $O$. So, one expects $\lambda(C,s)$ to encode plenty of geometry also for $s>1$. Note that we have the inequalities
\[
\lambda(C,s) \ \leqslant \ \sqrt{s} \ \leqslant \ \hat\mu(C,s) \ ,
\]
where the left-hand side is an equality if and only if the right-hand side is also such. From Conjecture~\ref{con:nn} this is expected to happen when $s$ is large enough and $C$ a sufficiently generic choice of a curve. Thus it becomes natural to ask about the shape of the convex set $\Delta_{C,s_+}$ when $\hat\mu(C,s)=\lambda(C,s)=\sqrt{s}$ does not happen.
\begin{Cor}\label{Cor:quadrilateral}
Under the assumptions above, one of the following happens: \\
\begin{inparaenum}
\item[(i)] $\hat\mu(C,s)=\lambda(C,s)=\sqrt{s}$, in which case $\Delta_{C,s_+}=\Delta_{\hat\mu, \hat\mu,\frac{\hat\mu}{s}}$, where $\hat\mu=\hat\mu(C,s)$.\\
\item[(ii)] $\hat\mu(C,s)>\sqrt{s}>\lambda(C,s)$, in which case $\lambda(C,s)=s/\hat\mu(C,s)$ and the convex polygon $\Delta_{C,s_+}$ is the quadrilateral $OABE$, where
\[
 O=(0,0), \quad A=(\lambda(C,s),0), \quad B=(\lambda(C,s),\lambda(C,s)/s), \quad E=(\hat\mu(C,s),c)\ ,
\]
for some $c\in[0,\frac{\hat\mu(C,s)}{s}]$. Hence, $\Delta_{C,s_+}$ is a triangle if and only if $c=0$ or $c=\frac{\hat\mu(C,s)}{s}$.
\end{inparaenum}
\end{Cor}

\begin{proof}
  By definition of $\hat\mu(C,s)$ and Lemma~\ref{lem:intersection}, for any effective divisor $Z$ in $X_{\frak K}$, not containing in its support any of the exceptional curves $A_i$, one has
\[
\hat\mu(C,s)\geqslant \frac{v_1(C,s)(Z)}{Z\cdot D}=\frac{Z\cdot B_s}{Z\cdot D} \ .
\]
So, $(\hat\mu(C,s)D -B_s)\cdot Z\geqslant 0$ for any such cycle $Z$. By  Lemma~\ref{lem:Bproperties}, we already know that $(\hat\mu(C,s)D -B_s)\cdot A_i\geqslant 0$ for any $i=1,\ldots, k$. Thus, the divisor $\hat\mu(C,s)D -B_s$ is nef and by Proposition~\ref{prop:nef} we get that $\lambda(C,s)\geqslant s/\hat \mu(C,s)$. When equality happens then we land in case $(1)$. Otherwise, if $(\hat \mu(C,s),c) \in \Delta_{C,s_+}$ for some $c\geqslant 0$, then this latter condition implies that $\Delta_{C,s_+}$ contains the convex hull of the points
\[
(0,0), \quad \left(\frac{s}{\hat\mu(C,s)},0\right),
\quad \left(\frac{s}{\hat\mu(C,s)},\frac{1}{\hat\mu(C,s)}\right), \quad
(\hat\mu(C,s),c) \ .
\]
Note that such a $c$ must exist, since the projection of $\Delta_{C,s_+}$ to the first axis is
$[0,\hat\mu(C,s)]$ as noted before. Since the area of this convex hull is $1/2$, it coincides with
$\Delta_{C,s_+}$, and one has $\lambda(C,s)= s/\hat \mu(C,s)$.
\end{proof}
\begin{Rem}
It is worth to note that Corollary~\ref{Cor:quadrilateral} takes place only because our ambient space is $\P^2$, especially due to properties like the Picard group of $\P^2$ is generated by a single class, whose associated line bundle is globally generated with self-intersection equal to $1$. One does not expect these phenomena to happen when we consider the valuations $v_{\pm}$ on any smooth projective surface $X$. But we do expect that some parts of the considerations about the infinitesimal picture developed in \cite{KL} to be true in this more general setup. For example, the constant $\lambda(C,s)$ should in some ways encode many interesting local positivity properties of the divisor class we are studying, as partially seen in Proposition~\ref{prop:nef}.
\end{Rem}

\begin{Exa}\label{616-2}
 Continuing with $s=48/7$ as in Example \ref{616}, and assuming $O$ is
 a general point of a curve $C$ of degree $d\geqslant 3$, we know from
 \cite[Theorem~C]{DHKRS} that $\hat\mu(C,48/7)=(1+48/7)/3=55/21$ and there is a
 unique curve $V$ with $v_1(C,48/7;V)/\deg V=55/21$, namely
 the unique cubic nodal at $O$ which has one of the branches $\gamma$
 at the node satisfying $(C,\gamma)_O=7$. Indeed,
 $V$ has multiplicity 2 at $O$, and (its strict transform) multiplicity $1$ at
 each center $p_2, \dots, p_7$, whereas it does not pass through any
 of the remaining centers $p_8, \dots, p_{13}$. So by~\eqref{eq:proximity}
 one has
 $v_1(C,48/7;V)=2+5+6/7=55/7$, which divided by $\deg (V)=3$ gives
 $\hat\mu(C,48/7)=55/21$.
 The Newton polygon of $V$ with respect to $C$ has three vertices, namely
 $(0,2)$, $(1,1)$ and $(8,0)$,
 showing that $v_1(C,s;V)=1+s$ for $s<7$, and so
 $v_+(C,48/7)(V)=(55/7,1)$. Therefore, the rightmost point of $\Delta_{C,48/7_+}$
 is the valuative point $(55/21,1/3)$.  Alternatively, $v_+(C,48/7)(V)$ can be computed
 from the pullback of $V$ to $X_{\frak K}$, which is
\begin{multline*}
 V^*=\tilde V+2B_1+B_2+\dots+B_7=\\
 A + 2A_1+3A_2+4A_3+5A_4+6A_5+7A_6+8A_7+\\
 15A_8+22A_9+31A_{10}+39A_{11}+47A_{12}+55A_{13}.
\end{multline*}
 Therefore the flag valuation applied to $V$ is
 $\nu_{Y_+}(V)=(55,8)$ (recall that one has \mbox{$p_+=A_7\cap A_{13}$})
 and the computation from Remark \ref{rem:flag2} gives
 \[v_+(C,48/7)(V)= \begin{pmatrix}
 \frac 1 q & 0 \\ -q' &  q
\end{pmatrix} \nu_{Y_+}(V)=
 \begin{pmatrix}
 \frac 1 7 & 0 \\ -1 &  7
\end{pmatrix}
\begin{pmatrix}
 55 \\ 8
\end{pmatrix}=
\begin{pmatrix}
 \frac{55}7 \\ 1
\end{pmatrix},
 \]
 consistent with the computation using the Newton polygon.
By Corollary \ref{Cor:quadrilateral}, the remaining vertices of $\Delta_{C,48/7_+}$
are $(0,0)$, $(144/55,0)$ and $(144/55,21/55)$.

The same computation applies to any $s\in (\phi^4,7)$, giving quadrilateral
bodies $\Delta_{C,s_+}$ with vertices
\[ (0,0), \quad \left(3-\frac{3}{s+1},0\right),
\quad \left(3-\frac{3}{s+1},\frac{3}{s+1}\right), \quad
\left(\frac{s+1}{3},\frac{1}{3}\right).\]
We leave the easy details to the reader.
\begin{figure}
\begin{centering}
\begin{tikzpicture}[scale=3,auto,label distance=0pt]
  \draw[thick,fill=gray!50!white] (0,0) --
  (144/55,0) --   (55/21,1/3) node (V) {} --
  (144/55,21/55) node [label=above:{$\left(\frac{144}{55},\frac{21}{55}\right)$}] {} -- cycle;
  \draw[<-] (V) to ++(1/9,0) node [label=right:{$\left(\frac{55}{21},\frac{1}{3}\right)$}]  {};
  \draw (1/4,1/4) node {$\Delta_{C,48/7_+}$};
  \draw[thick,fill=gray!50!white] (0,-1) --
  (21/8,-1) --   (8/3,1/3-1) node (V) {} --
    (21/8,3/8-1)  node [label=above:{$\left(\frac{21}{8},\frac{3}{8}\right)$}] {} -- cycle;
  \draw[<-] (V) to ++(1/8,0) node [label=right:{$\left(\frac{8}{3},\frac{1}{3}\right)$}]  {};
  \draw (1/4,1/4-1) node {$\displaystyle\lim_{s \to 7^-}\Delta_{C,s_+}$};
  \draw[thick,fill=gray!50!white] (0,-2) --
  (21/8,-2) --   (8/3,-2) node [label=right:{$\left(\frac{8}{3},0\right)$}]  {} --
    (21/8,3/8-2)  node [label=above:{$\left(\frac{21}{8},\frac{3}{8}\right)$}] {} -- cycle;
  \draw (1/4,1/4-2) node {$\Delta_{C,7_+}$};
\end{tikzpicture}
%
%
\end{centering}
\caption{\small{ The \NO body computed in Example \ref{616-2}, on top.
The difference $55/21-144/55=1/1155$ is so small that the quadrilateral
looks like a triangle. As $s$ grows in the interval $(\phi^4,7)$, the vertex
$(\frac{s+1}3,\frac13)$ sticks out further right from the two vertices
with first coordinate $3-\frac{3}{s+1}$, so that at the limit, the quadrilateral
nature is clearly seen. At $s=7$ the quadrilateral mutates into a triangle,
shown in the bottom picture.}}
\end{figure}
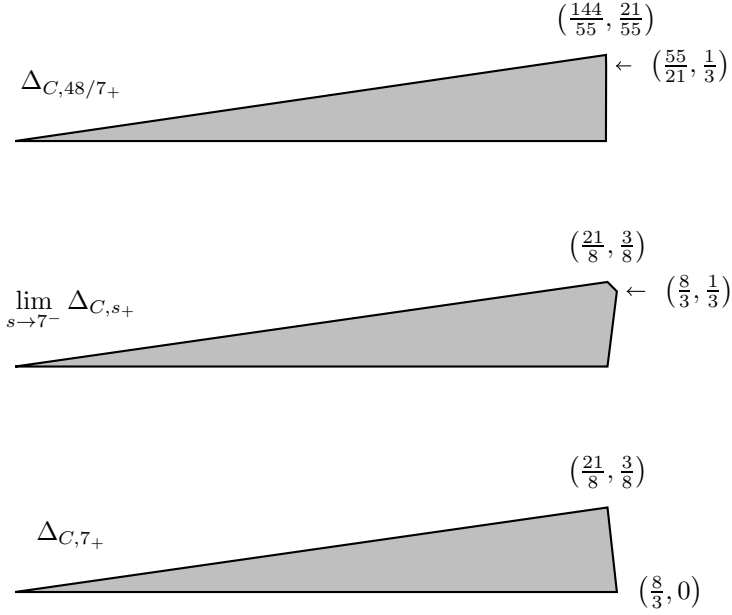

\end{Exa}

\subsection{Large $s$ on curves of fixed degree}
As we have seen in Corollary~\ref{Cor:quadrilateral}, the convex set $\Delta_{C,s_+}$ can be either a rectangular triangle or a quadrilateral, for any rational number $s\geqslant 1$ and any degree $d=\deg(C)$. Interestingly enough, when $s\gg d$ one can say more. More precisely, we have the following theorem:
\begin{Teo}\label{thm:larges}
 Let $C\subset \P^2$ be a plane curve of degree $d$ and $D\subseteq \P^2$ a line. Then $\Delta_{C,s_+}=\Delta_{d,\frac sd,\frac 1d}$ for any $s\geqslant d^2$. In particular, one has
 \begin{equation}\label{eq:stab}
 \Delta_{(C,O)}(D) \ = \ \Delta_{\frac{1}{d},d} \ = \
\begin{pmatrix}
 0 & 1 \\ 1 & -s
\end{pmatrix}\Delta_{C,s_+}
 \end{equation}
\end{Teo}

\begin{proof}
First, the point $(0,0)$ belongs to $\Delta_{C,s_+}$, as it is given by the valuation of any polynomial not vanishing at $O$. Also, $(\frac sd,\frac 1 d)\in\Delta_{C,s_+}$, being the valuation of an equation of $C$, as we saw in the proof of Proposition \ref {deriv_val} (see also Remark \ref {rem:flag2}). It remains to show that $(d,0)\in\Delta_{C,s_+}$. This will be enough, since we will have the containment $\Delta_{d,\frac sd,\frac 1d}\subseteq \Delta_{C,s_+}$ and the two coincide because both have area $1/2$.

For the latter condition, fix any $\epsilon>0$ and let $f_\epsilon$ be the polynomial as in Corollary~\ref{vertex_approach}, satisfying~(i) and~(ii). Consider the function $s\mapsto v_{+}(C,s)(f_{\epsilon})\in\R^2$ for~\mbox{$s\geqslant d^2$}.  By Proposition~\ref{lem:trop}, the first coordinate of this function is non-decreasing and concave. Thus, the second coordinate is decreasing by concavity of the first. Now, these two remarks, together with the properties we know for $v_{+}(C,s)(f_{\epsilon})$ when $s=d^2$ from Corollary~\ref{vertex_approach}, imply that the limit
\[
\lim_{\epsilon\rightarrow\infty} \frac{v_+(C,s)(f_{\epsilon})}{\deg(f_{\epsilon})} \ = \ (t_s,0), \textup{ for any } s\geqslant d^2 \ ,
\]
where $t_s\geqslant d$. In particular, $(t_s,0)\in\Delta_{C,s_+}$ and since the origin is contained in this convex set, then this implies that $(d,0)\in\Delta_{C,s_+}$.

By Example \ref {ex:p} we have $\Delta_{(C,O)}(D)=\Delta_{\frac 1d,d}$, thus
the final assertion follows. \end{proof}

\begin{Rem}
 With similar arguments as in the proof of Theorem \ref {thm:larges}, one proves that $\Delta_{C,s_-}=\Delta_{d,\frac sd,-\frac 1d}$ and
  \[
 \Delta_{(C,O)}=
\begin{pmatrix}
 0 & -1 \\ 1 & s
\end{pmatrix}\Delta_{C,s_-}
 \]
\end{Rem}

\begin{Cor}\label{cor:conj} If $C$ is a plane curve of degree $d$, then
\[
\hat \mu(C,s)=\frac sd ,\quad \forall s\geqslant d^ 2 \ .
\]
Hence $v_1(C,s)$ is minimal (resp. not minimal) for $s=d^ 2$ (resp. for $s>d^2$) (see~\S3.4 for definition).
\end{Cor}

\begin{Rem} Corollary \ref {cor:conj} seems to be in contrast with Conjecture \ref{con:nn}, which in reality is not the case. Conjecture \ref{con:nn} applies only for a sufficiently general choice of $C$. Given $s$, this requires the degree of $C$ to be large enough with respect to $s$. In other words, if $s\geqslant d^ 2$ then $C$ is not sufficiently general.
\end{Rem}

\begin{Rem} Lemma~8 from \cite{Roe} implies that the right-hand side of \eqref {eq:stab} \emph{converges} to the left-hand side for $s\to +\infty$. So, Theorem \ref {thm:larges} makes this statement more precise, i.e., in fact the two bodies are equal for $s\geqslant d^ 2$.\end{Rem}

\subsection{Mutations and supraminimal curves}
\label{sec:supraminimal}
By fixing  $C$, the goal of this section is to study $\Delta_{C,s_+}$ as a \emph{function} of $s\in [1,+\infty)$, i.e. by walking along an arc of $\mathcal {QM}$ away from the root.  By Theorem \ref {thm:larges}, the picture is well understood for $s\geqslant d^ 2$, so it remains to study the case when $s\in (1,d^ 2)$. Note that the origin $(0,0)$ is always a vertex of $\Delta_{C,s_+}$. The remaining vertices of $\Delta_{C,s_+}$ will be called \emph{proper} and their behaviour is the focus of this subsection.

\begin{Def}\label{def:mut} We say that $\Delta_{C,s_+}$ is \textit{continuous} at $s_0\in (1,d^ 2)$ if, for every $\epsilon>0$ there exists a $\delta>0$ such that for all
$s$ with $|s-s_0|<\delta$, every vertex $p$ of $\Delta_{C,s_+}$ is near the
boundary of $\Delta_{C,s_{0+}}$, i.e. $\textup{distance}(p, \partial \Delta_{C,s_{0+}})<\epsilon$.

 If $\Delta_{C,s_+}$ is not continuous for some $s_0\in (1,d^ 2)$, then we say that $\Delta_{C,s_+}$ presents a \emph{mutation} at $s_0$ (or \emph{mutates} at $s_0$). Also, $\Delta_{C,s_+}$ \emph{depends linearly on $s$} in an interval $I\subseteq(1,d^ 2)$, if the number of proper vertices of $\Delta_{C,s_+}$ is the same for all $s\in I$ and the coordinates of the vertices of $\Delta_{C,s_+}$ are affine functions of $s$ in $I$.
\end{Def}

As we will see, a standard reason for mutation,
taking place between intervals of linearity of $\Delta_{C,s_+}$,
is non-minimality of $v_1(C,s)$.
Moreover mutations  may behave differently according to whether $O$ is sufficiently general on $C$ or not.

\begin{Def} We say that an irreducible curve $V$ containing $O$, with equation $f=0$, \emph{computes} $\hat\mu(C,s)$ via $v_1(C,s)$ if $v_1(C,s;f)=\deg (f) \cdot  \hat \mu(C,s)$.

If $v_1(C,s)$ is non-minimal, there exists $V$ computing $\hat \mu(C,s)$ via $v_1(C,s)$. Hence
$$v_1(C,s;f)=\deg (f) \cdot  \hat \mu(C,s)>\deg (f) \cdot\sqrt s$$
(see \cite [Lemma 5.1]{DHKRS}). Such  curves are called \emph{supraminimal} for $v_1(C,s)$.
\end{Def}

\begin{Rem}\label{rem:acc} The proof of \cite [Lemma 5.1] {DHKRS} shows that if $V$ computes $\hat \mu(C,s)$, in particular if it is supraminimal for $v_1(C,s)$, then there is no (other) supraminimal curve at $s$, and that if $V$ and $W$ both compute $\hat \mu(C,s)$ and are distinct, then $(V,W)_0=\deg(V)\cdot \deg(W)$.

If $V$, with equation $f=0$, computes $\hat \mu(C,s)$, then the valuative points of $v_+(C,s)$ corresponding to $f$
(i.e., the one-sided limits of $v_+(C,s)(f)$ with respect to $s$)
are  \emph{rightmost} vertices of $\Delta_{C,s_+}$. Note that there are two such vertices, if $\partial v_1(C,\,\cdot\,;f)$ is discontinuous at $s$, in which case there is a mutation of $\Delta_{C,s_+}$ at $s$, or there is only one such vertex.
\end{Rem}

\begin{Exa}
Corollary \ref {cor:conj} tells us that $C$  of degree $d$ is supraminimal for all $v_1(C,s)$ with $s>d^2$: we consider this as a \emph{trivial} case of supraminimality.

There is no supraminimal curve for the $O$-adic valuation  $v_1(C,1)$. So, given $C$, non-trivial supraminimal curves for $v_1(C,s)$ may occur only for $s\in (1,d^ 2)$.

\end{Exa}

\begin{Teo}\label{thm:bb} Let $C$ be a smooth curve through $O$ and let $V$ be a curve, different from $C$, with equation $f=0$. Then:\\
\begin{inparaenum}
\item [(i)] the set of points $s\in [1,+\infty)$ such that $V$ is supraminimal for $v_1(C,s)$ is open; \\
\item [(ii)] if $V$ is supraminimal for $v_1(C,s)$ for all $s\in(a,b)$ but not at $a$ and $b$, then
\[
\Delta_{C,a_+}=\Delta_{\sqrt{a},\sqrt{a},1/\sqrt{a}} \ \ \Delta_{C,b_+}=\Delta_{\sqrt{b},\sqrt{b},1/\sqrt{b}} \ ,
\]
and there is some point $\sigma\in (a,b)$ of discontinuity of the derivative of $v_1(C,s)$. Furthermore, there is some branch $\gamma\in \mathcal B(V)$ with Puiseux expansion starting as $y=ax^\sigma+\ldots$, such that $-\frac 1 \sigma=\slo(\mathfrak l)$ with $\mathfrak l=\varphi_V(\gamma)\in \mathfrak E(f)$; \\
\item  [(iii)] $\Delta_{C,s_+}$ mutates at the finitely many points $\sigma\in (a,b)$ as in (ii).
\end{inparaenum}
\end{Teo}

\begin{proof} Part (i) follows from the continuity of $v_1(C;f)(s)$ and $\sqrt s$.\medskip

Let us prove~(ii). There is no supraminimal curve for $v_1(C,a)$ and $v_1(C,b)$.
Indeed, suppose $W$ is supraminimal for $v_1(C,a)$ (the same argument works for  $v_1(C,b)$). Then, by part (i), $W$ is supraminimal for  $v_1(C,s)$ with $s$ in a neighborood of $a$. But $V$ is also supraminimal for  $v_1(C,s)$ with $s$ in a right neighborood of $a$. By the uniqueness of supraminimal curves, we have $V=W$, against the assumption that $V$ is not supraminimal for $v_1(C,a)$. Thus $\hat \mu(C,s)=\sqrt s$ for $s=a,b$. Since $v_1(C;f)(s)>\deg(f)\cdot \sqrt s$ for $s\in (a,b)$ and $v_1(C;f)(s)$, as a function of $s\in [1,+\infty)$, is a tropical polynomial (see Proposition \ref {lem:trop}), certainly there is some point $\sigma\in (a,b)$ of discontinuity for its derivative. Moreover $v_1(C;f)(s)=\deg(f)\cdot \sqrt s$ for $s=a,b$.
The rest of (ii) follows from the discussion in \S\ref {ssec:NP}. \medskip

To show~(iii) note that the mutation of $\Delta_{C,s_+}$ at the points $s\in (a,b)$ as in (ii) depends on the discontinuity of  $\partial v_1(C;f)$ there (see Remark \ref {rem:acc}).
\end{proof}

\begin{Pro}
  If $\hat\mu(C,s_0)=\sqrt {s_0}$ then $\Delta_{C,s_+}$ is continuous at $s_0$.
\end{Pro}
\begin{proof}
By Corollary~\ref{Cor:quadrilateral}, for every $s$ there are inclusions
\[ \Delta_{\frac{s}{\hat\mu(C,s)},\frac{s}{\hat\mu(C,s)},\frac{1}{\hat\mu(C,s)}}
\subseteq \Delta_{C,s_+}
\subseteq
\Delta_{\hat\mu(C,s),\hat\mu(C,s),\frac{\hat\mu(C,s)}s}\]

Since $\hat\mu(C,s)$ is a continuous function of $s$,
it follows that for every $\epsilon>0$
there is $\delta>0$ such that for $|s-s_0|<\delta$,
\[ \Delta_{\sqrt{s_0}-\epsilon,\sqrt{s_0}-\epsilon,\frac{\sqrt{s_0}-\epsilon}{s}}
\subseteq \Delta_{C,s_+}
\subseteq
\Delta_{\sqrt{s_0}+\epsilon,\sqrt{s_0}+\epsilon,\frac{\sqrt{s_0}+\epsilon}{s}}\]
and the claim follows,
since $\Delta_{C,s_0+}=\Delta_{\sqrt{s_0},\sqrt{s_0},1/\sqrt{s_0}}$.
\end{proof}
The mutations described in Theorem~\ref{thm:bb}(iii), can be called \emph{supraminimality mutations}.
\begin{Cor}
Any mutation of $\Delta_{C,s_+}$  is supraminimal.
\end{Cor}

For general choices of $O$ and $C$, all known supraminimal curves
are $(-1)$-curves.
It would be interesting to explore the behavior of $\Delta_{C,s_\pm}$
on surfaces different from~$\P^2$, or for non-quasimonomial valuations
(i.e., allowing singular $C$, and following arcs in the whole valuative
tree $\mathcal V$ rather than only $\mathcal{QM}$). It is tempting
to conjecture that mutations in general should be supraminimal
and related to extremal rays in (some) Mori cone.

\subsection{Explicit computations}

In this section we compute the \NO bodies in the range in which $\hat\mu(C,s)$ is known (see \cite {DHKRS} and \S\ref {ssec:mu}).

\subsubsection{The line case}

This case is an immediate consequence of Theorem~\ref{thm:larges}, which we state here for completeness.

\begin{Pro}
 If $C$ is a line, then for every $s\geqslant 1$ one has
  $$\Delta_{C,s_+}=\Delta_{1,s,1},$$
  hence $\Delta_{C,s_+}$ depends linearly on $s$ and there are no mutations.
 \end{Pro}

\subsubsection{The conic case}

\begin{Pro}\label{prop:conic}
 If $C$ is a conic, then
 \[
 \Delta_{C,s_+}=
\begin{cases}
 \Delta_{1,s,1} & \text{if } 1\leqslant s<2 \\
  \Delta_{2,\frac s2,\frac 12}  & \text{if }   s\geqslant 2 \ .
\end{cases}
 \]
So, there is only one mutation at $s=2$, and $v_1(C,s)$ is minimal only when $s=1, 4$.
 \end{Pro}
\begin{proof}

The case $s=1$ is trivial. So, assume first that $s\in (1,2)\cap\Q$. By Remark~\ref{rem:flag2}, we know that $(1,0)\in\Delta_{C,s_+}$. It remains to show that $(s,1)\in\Delta_{C,s_+}$. Note that the only free points in the cluster $\mathfrak K=\mathfrak K_{v_+(C,s)}$ are $p_1=O$ and $p_2$. The line through $p_1$ and $p_2$, i.e., the tangent line to $C$, has equation $y=0$ and behaves exactly like $C$ with respect to $\mathfrak K$. Thus, again by Remark~\ref{rem:flag2}, we have
\[
\frac {v_+(C,s;y)}{\deg(y)}=v_+(C,s;y)=(s,1)\ ,
\]
implying that $(s,1)\in \Delta_{C,s_+}$.

Take now $s\in [2,4)\cap\Q$ and use the notation of \S\ref {flagmodel}.  By Remark \ref {rem:flag2}, we have that $(\frac s2,\frac 12)\in\Delta_{C,s_+}$, given by the valuation of a local equation of $C$. It remains to show that $(2,0) \in \Delta_{C,s_+}$. Let $L$ on $X_\mathfrak K$ be the total transform of the tangent line to $C$ at $O$. Note that $L-B_1-B_2$ does not contain any of the other exceptional curves. Thus, arguing as in Lemma~\ref{lem:loc}(ii), it is easy to see that $L$ contains $A_k$ with  multipliticy $2q$ and $L-2qA_k$ passes through $x_+$ with multiplicity $2q''$.  Hence $(2q,2q')\in \Delta_{Y_+}$ and, by Remark \ref {rem:flag2}, this implies that $(2,0) \in \Delta_{C,s_+}$.

Finally, by Theorem \ref {thm:larges}, the assertion holds for $s\geqslant 4$.
\end{proof}

\subsubsection{The higher degree case}

The case in which $\deg(C)\geqslant 3$ is more interesting, since it gives rise to infinitely many mutations of the \NO body. Recall the notation $\{F_i\}_{i\in \mathbb{Z}_{\geqslant 0}\cup \{-1\}}$ for the sequence of the Fibonacci numbers,
and $\phi$ for the golden ratio (see \S\ref {ssec:mu}).

\begin{Pro}\label{prop:Or}
 For each odd integer $i\geqslant 5$, there exists a rational curve $C_i\subseteq\P^2$ of degree $F_{i}$
 with a single cuspidal (i.e., unibranch) singularity at $O$ and  characteristic exponent $\frac {F_{i+2}} {F_{i-2}}\in (6,7)$,  whose six free points infinitely near to $O$ are in general position. Let $C_1$ be a line (of degree $F_1$ with \emph{characteristic exponent} $\frac {F_{3}} {F_{-1}}=2$) and $C_3$ be a conic (of degree $F_3$ with \emph{characteristic exponent} $\frac {F_{5}} {F_{1}}=5$). All these curves are $(-1)$-curves in their embedded resolution (i.e., after blowing up the appropriate points of the cluster of centres determined by the characteristic exponent).

 If $C$ is a general curve with $\deg (C) \geqslant 3$ through the origin $O$, then the curve~$C_i$, with equation $f_i=0$, through $O$ and the first six infinitely near points to $O$ along~$C$,  satisfies
\[\hat \mu(C,s)=\frac{v_1(C,s;f_i)}{\deg (f_i)}=
\begin{cases}
 \frac{F_{i-2}}{F_{i}}s & \text{if }s \in \left[\frac{F_{i}^2}{F_{i-2}^2},\frac{F_{i+2}}{F_{i-2}}\right),\\
\frac{F_{i+2}}{F_{i}} & \text{if }s \in \left[\frac{F_{i+2}}{F_{i-2}},\frac{F_{i+2}^2}{F_{i}^2}\right] .
\end{cases}
\]
Thus $C_i$ is supraminimal for $v_1(C,s)$ for
 $s\in \left(\frac{F_{i}^2}{F_{i-2}^2},\frac{F_{i+2}^2}{F_{i}^2}\right)$ and any odd $i\geqslant 1$.

\end{Pro}
\begin{proof} The existence result is in \cite [Theorem C, (a) and (b)] {Ore02}. The rest of the assertion is \cite  [Proposition 5.5] {DHKRS}.
\end{proof}

\begin{Pro}\label{prop:cubic} If $\deg(C)\geqslant 3$ and $O\in C$ is a general point, then:\\\begin{inparaenum}
\item [(i)] one has
 \begin{equation}\label{eq:cases}
 \Delta_{C,s_+}=
\begin{cases}
 \Delta_{1,s,1} & \text{if } 1\leqslant s<2 \\
  \Delta_{2,\frac s2,\frac 12}  & \text{if }   2\leqslant s \leqslant 5\\
   \Delta_{\frac 52,\frac {2s}5, \frac 25}  & \text{if }  5 < s\leqslant 6+\frac 14,\\
 \end{cases}
\end{equation}
hence $ \Delta_{C,s_+}$ mutates at $s=2, s=5$,
and depends linearly on $s$ between mutations;\\
\item [(ii)] for $s\in [6+\frac {1}4, \phi^4)$ one has
\[
 \Delta_{C,s_+}=
 \Delta_{\frac{F_{i}}{F_{i-2}},\frac {F_{i-2}}{F_i}s,\frac{F_{i-2}}{F_{i}}} \text{if }s \in \left[\frac{F_{i}}{F_{i-4}},\frac{F_{i+2}}{F_{i-2}}\right);
\]
i.e.,  $\Delta_{C,s_+}$ mutates at $\frac {F_{i+2}} {F_{i-2}}$, for all odd integers $i\geqslant 5$
(these mutations agree with part (ii) of Theorem \ref {thm:bb}),
and depends linearly on $s$ between mutations;\\
 \item [(iii)] for $s\in (\phi^4, 7)$, the \NO body is the quadrilateral with vertices
$$  (0,0), (\frac{3s}{1+s},0), (\frac {1+s}3,\frac 13), (\frac {3s}{1+s},\frac 3{1+s});$$
   \item [(iv)] for $s\in (7, 7+\frac 19)$  one has
 $$\Delta_{C,s_+}=\Delta_{\frac 83,\frac 38 s, \frac 38}.$$ \end{inparaenum}
 Accordingly, there is a mutation at $s=7$.
 \end{Pro}
 \begin{proof}
   All the claims follow from Corollary \ref{Cor:quadrilateral} taking into account the computations of $\hat\mu(C,s)$ from
\cite{DHKRS}, see Remark \ref{muhatDHKRS}. The only problem, when applying Corollary~\ref{Cor:quadrilateral}, is to know where the vertex farthest to the right of $\Delta_{C,s_+}$ lies on the line $t=\hat\mu(C,s)$,
i.e., we have to compute the number $c$ appearing in Corollary~\ref{Cor:quadrilateral}(ii). This is given by the valuation of the curve $C_1$ or $C_3$ in case~(i), of the curve $C_i$ for any
$i\geqslant 5$ in case~(ii)
(where both examples of curves are retrieved from Proposition~\ref{prop:Or}), 
and the cubic curve $D_1$ from \cite[Table~5.1]{DHKRS}
for both~(iii) and~(iv).
 \end{proof}

\begin{Rem} If $\deg(C)=d$, Proposition \ref {prop:Or} leaves an unkown interval $[7+\frac 19,d^ 2)$ whereas for $s\in [d^ 2,+\infty)$
the \NO body is known by Theorem \ref {thm:larges} and there are no mutations there. Conjecturally, the same should happen for
$$s\geqslant 8+\frac 1 {36}$$
(see Conjecture \ref {con:nn}) if $O$ is a general point of $C$.
\end{Rem}

\begin{Cor}\label{corollary:final}
   Assume $\deg(C)\geqslant 3$ and $O$ is a general point of $C$. Then for $s\leqslant 7$,
   the Newton--Okounkov body $\Delta_{C,s_+}$
lies in the half-plane $t+u\leqslant 3$. Moreover,
\begin{enumerate}
\item[(i)] If $s\in [1,\phi^4)\cup [7,7+\frac{1}{9})$, then $\Delta_{C,s_+}$ is a triangle whose vertices are
valuative.
\item[(ii)] If $s\in(\phi^4,7)$, then $\Delta_{C,s_+}$ is a quadrilateral with at least one vertex being
a non-valuative point.
\end{enumerate}
 \end{Cor}
\begin{proof}
All claims follow from Proposition~\ref{prop:cubic}, except the fact that $\Delta_{C,s_+}$ has a non-valuative vertex when $s\in (\phi^4,7)$, and that $(\frac{3s}{8},\frac{3}{8})$ is valuative when $s\in [7,\frac{64}{9})$.

First, assume that $(\frac{3s}{1+s},\frac{3}{1+s})$ is valuative for some $s <7$.
This means that there is a polynomial $f$ of degree $d$ with
\[ v_1(C,s;f)=\frac{3s}{1+s}d,\quad\text{and}\quad \partial_+v_1(C;f)(s)=\frac3{1+s}d. \]
As $v_1(C;f)$ is piecewise linear as a function of $s$, then for small enough $\epsilon>0$,
\[
v_1(C,s+\epsilon;f)=\frac{3s}{1+s}d+\frac3{1+s}d\epsilon,\quad\text{and}
\quad \partial_+v_1(C;f)(s+\epsilon)=\frac3{1+s}d.
\]
In particular, this implies that
\[
\left( \frac{3s}{1+s}+\frac{3\epsilon}{1+s}, \frac3{1+s} \right) \in \Delta_{C,(s+\epsilon)_+}.
\]
But this contradicts that for $s+\epsilon<7$, the Newton--Okounkov body $\Delta_{C,s_+}$ lies in the half-plane $t+u\leqslant 3$. So, $(\frac{3s}{1+s},\frac{3}{1+s})$ is a non-valuative vertex.

Finally, $(\frac{3s}{8},\frac{3}{8})$ is valuative because there is a unique curve $V$ of degree 24
whose Newton polygon with respect to $C$ has vertices $(0,9)$ and $(64,0)$. Indeed,
let $\frak K$ be the cluster of centers of $v_1(C,\frac{64}{9})$, which consists of 8
free points followed by~8 satellites, each proximate to its predecessor and to
$p_7$ (the continued fraction of~$\frac{64}{9}$ is $[7;9]$). Then $V$ has multiplicity
9 at each of $p_1, \dots, p_7$ and multiplicity 1 at $p_8, \dots, p_{16}$.

The curve $V$ has genus 1 and is obtained in this way. Consider the Cremona transformation $\omega$ determined by the homaloidal system of curves of degree 8 with triple points at a cluster $\mathfrak C$ of seven general infinitely near, free, base points (this Cremona transformation appears in the construction of the curves $C_i$ in
Proposition~\ref{prop:Or}, see \cite[proof of Theorem C]{Ore02}).

There is a unique cubic curve $\Gamma$ with a double point at the first  point of $\mathfrak C$  and passing simply through the remaining six points of $\mathfrak C$. This curve is contracted to a point by $\omega$.

Let $x\in \Gamma$ be a general point. There is a pencil $\mathcal P$ of cubics having intersection multiplicity 8 with $\Gamma$ at $x$. Then  $\mathcal P$ has 9 base points, 8 are given by the cluster formed by
$x$ and by the 7 infinitely near points to $x$ along $\Gamma$, and  there is a further base point $y\in \Gamma$. The general curve of $\mathcal P$ is irreducible, and its image via $\omega$ is the required curve $V$, which has genus 1.
\end{proof}

\begin{Rem}
It is somewhat mysterious that in the case $(ii)$ of Corollary~\ref{corollary:final}
one has a vertex of $\Delta_{C,s_+}$ that is not valuative, taking into account that  for $s<7$, $s\in \Q$, the Mori cone of $X_{\frak K}$ is polyhedral (see \cite{DHKRS}).
\end{Rem}

\bibliographystyle{amsplain}

\end{document}